\documentclass[reqno]{amsart}

\usepackage{amsmath,amsthm,amssymb,mathrsfs}
\usepackage{enumerate}
\usepackage{enumitem}
\usepackage{bbm}
\usepackage{hyperref}
\usepackage{cleveref}
\usepackage[textsize=tiny]{todonotes}
\usepackage{color}
\usepackage[all]{xy}
\usepackage{proofread}

\newtheorem{theorem}{Theorem}[section]
\newtheorem{lemma}[theorem]{Lemma}

\newtheorem{proposition}[theorem]{Proposition}

\theoremstyle{definition}
\newtheorem{definition}[theorem]{Definition}
\newtheorem{remark}[theorem]{Remark}
\newtheorem{example}[theorem]{Example}

\DeclareMathOperator{\rk}{rk}
\DeclareMathOperator{\dlog}{dlog}

\newcommand{\dd}{\textrm{d}}

\newcommand{\Z}{\mathbb{Z}}
\newcommand{\Q}{\mathbb{Q}}
\newcommand{\C}{\mathbb{C}}
\newcommand{\A}{\mathcal{A}}
\newcommand{\exA}{\mathcal{B}}
\newcommand{\poset}{\mathcal{L}}
\newcommand{\oomega}{\overline{\omega}}

\newcommand{\oeta}{\overline{\eta}}

\DeclareMathOperator{\gr}{gr}
\DeclareMathOperator{\FF}{\mathcal{F}}

\newcommand{\uno}{\mathbf{1}}
\newcommand{\bj}{\mathbf{j}}
\newcommand{\bk}{\mathbf{k}}
\newcommand{\bm}{\mathbf{m}}
\newcommand{\bn}{\mathbf{n}}

\newcommand{\cf}{cf.\ }
\newcommand{\ie}{i.e.\ }

\def\MR#1{}

\def\XE{\Lambda(E)}

\begin{document}
	
\title[OS-type presentations for the cohomology of toric arrangements]{Orlik-Solomon-type presentations for the cohomology algebra of toric arrangements}

\author[F. Callegaro]{Filippo Callegaro}
\address{Filippo Callegaro \newline Universit\'a di Pisa\\ Dipartimento di Matematica\\Largo Bruno Pontecorvo 5, 56127 Pisa\\ Italia}\email{callegaro@dm.unipi.it}

\author[M. D'Adderio]{Michele D'Adderio}
\address{Michele D'Adderio \newline Universit\'e Libre de Bruxelles (ULB)\\D\'epartement de Math\'ematique\\ Boulevard du Triomphe, B-1050 Bruxelles\\ Belgium}\email{mdadderi@ulb.ac.be}

\author[E. Delucchi]{Emanuele Delucchi}
\address{Emanuele Delucchi \newline Universit\'e de Fribourg\\D\'epartement de math\'ematiques\\ Chemin du Mus\'ee 23, CH-1700 Fribourg\\ Switzerland}\email{emanuele.delucchi@unifr.ch}

\author[L. Migliorini]{Luca Migliorini}
\address{Luca Migliorini \newline Universit\'a di Bologna\\Dipartimento di Matematica\\ Piazza di Porta S.Donato 5, Bologna\\ Italia}\email{luca.migliorini@unibo.it}

\author[R. Pagaria]{Roberto Pagaria}
\address{Roberto Pagaria \newline Scuola Normale Superiore\\ Piazza dei Cavalieri 7, 56126 Pisa\\ Italia}\email{roberto.pagaria@sns.it}

\begin{abstract}

We give an explicit presentation for the integral cohomology ring of the complement of any arrangement of level sets of characters in a complex torus (alias ``toric arrangement''). Our description parallels the one given by Orlik and Solomon for arrangements of hyperplanes, and builds on De Concini and Procesi's work on the rational cohomology of unimodular toric arrangements. As a byproduct we extend Dupont's rational formality result to  formality over $\mathbb Z$.

The data needed in order to state the presentation of the rational cohomology is fully encoded in the poset of connected components of intersections of the arrangement.	
\end{abstract}

\maketitle


\tableofcontents

\section{Introduction}

The topology of the complement of an arrangement of hyperplanes in a complex vector space is a classical subject, whose study received considerable momentum form early work of Arnol'd and Brieskorn (e.g., \cite{Arnold1969,Brieskorn1973}) motivated by applications to the theory of braid groups and of configuration spaces. A distinguishing trait of this research field is the deep interplay between the topological and geometric data and the arrangement's combinatorial data, here usually understood to be the arrangement's {\em matroid}, a combinatorial abstraction of the 
linear dependencies among the hyperplanes' defining forms.
%
%
A milestone in this direction is the presentation of the complement's integral cohomology algebra given by Orlik and Solomon \cite{OrlikSolomonOriginal}, building on work of Arnol'd and Brieskorn. As we will explain below, this presentation \
is fully determined
by the combinatorial (matroid) data and thus such an algebra can be associated with any matroid. Over the years, {\em Orlik-Solomon algebras} of general matroids have attracted interest in their own right \cite{YuzvinskiSurvey2001}.

\medskip

In the wake of De Concini, Procesi and Vergne's work  on the connection between partition functions and splines \cite{deConciniProcesiVergne} 
came a renewed interest in the study of complements of arrangements of subtori in the complex torus -- a class of spaces which had already been considered by Looijenga in the context of moduli spaces \cite{Looijenga}. Following \cite{DeConciniProcesiToricArr} we call such objects {\em toric arrangements}. In \S\ref{ss:HA-OS} below we will briefly outline the state of the art on the topology of toric arrangements. 
This research direction was spurred particularly by the seminal work of De Concini and Procesi \cite{DeConciniProcesiToricArr} which foreshadowed 
as rich an interplay between topology and combinatorics as is the case for hyperplane arrangements.

A crucial aspect that emerged in \cite{DeConciniProcesiToricArr} and was confirmed by subsequent research in the topology of toric arrangements is that the matroid data naturally associated with every toric arrangement is not fine enough to encode meaningful geometric and topological invariants of the arrangement's complement. 
The quest for a suitable enrichment of matroid theory has been pursued from different points of view, i.e., by modeling the algebraic-arithmetic structure of the set of characters defining the arrangement \cite{DM,BM,FinkMoci}   or by studying the properties of the pattern of intersections \cite{DelucchiRiedel,Pagaria2017} (see \S\ref{ss:TA-TOP} and \S\ref{sec:AM}).

\medskip

In this paper we provide an Orlik-Solomon type presentation for the cohomology algebra of an arbitrary toric arrangement, generalizing De Concini and Procesi's work on the unimodular case.
Our presentation with rational coefficients is fully determined by the intersection pattern.
This presentation holds also for the integral cohomology algebra, but, in this case, it is not determined by the intersection pattern.
In order to be able to state our results we provide some background.






%
%
%

\subsection{Arrangements of hyperplanes and Orlik-Solomon Algebras}\label{ss:HA-OS}
We start with a \emph{(central) hyperplane arrangement}, \ie a finite set  $\mathcal A = \{H_\lambda\}_{\lambda \in E}$ of codimension one linear subspaces in a complex vector space $V\simeq \mathbb{C}^n$.
The space $M(\mathcal A):= V\setminus \cup \mathcal A$ is in a natural way an affine complex variety, hence its cohomology (over $\mathbb C$) is computed by the {\em algebraic} de Rham complex, as
the quotient of the group of closed algebraic forms modulo that of exact algebraic ones (by Grothendieck's algebraic de Rham theorem \cite{Grothendieck}).

We choose vectors 
$\{a_\lambda \}_{\lambda \in E} \subset V^*$ such that $H_\lambda = \operatorname{ker} a_\lambda$ and consider the free exterior algebra $\Lambda_E$ over $\Z$ generated by the symbols $\{e_\lambda \}_{\lambda\in E}$. In $\Lambda_E$ we define an ideal 
as follows: for every  subset $A:=\{a_{\lambda_1}, \cdots ,a_{\lambda_r}\} \subset \{a_\lambda \}_{\lambda\in E}$  of linearly {\em dependent} vectors, we set
\[
\partial e_A:=\sum_{i=1}^r(-1)^{i-1}e_{\lambda_1}\cdots \widehat{e_{\lambda_i}}\cdots
e_{\lambda_r}
\]
and let $J_E$ be the ideal generated by the $\partial e_A$'s, where $A$ runs over all linearly dependent subsets of $E$. 


The quotient algebra $\Lambda_E / J_E$
is called the \emph{Orlik-Solomon algebra} of the arrangement.
The theorem of Orlik and Solomon states that the map $\Lambda_E \to H^*\left(M(\mathcal A), \Z\right)$ sending
$e_\lambda$ to the differential form $\frac{1}{2\pi i}\dlog a_\lambda$ factors to an algebra isomorphism
\begin{equation*}
	\Lambda_E / J_E \stackrel{\simeq}{\longrightarrow} H^*\big{(}V \setminus \bigcup_{\lambda \in E} H_\lambda, \Z\big{)}.
\end{equation*}
Two consequences of this fact are:
\begin{enumerate}
	\item 
	$H^*\left(M(\mathcal A), \Z \right)$ is generated in degree one;
	\item
	the integral ring structure depends only on the structure of the family of linearly dependent subsets of $\{a_\lambda\}_{\lambda \in E}$.
\end{enumerate}

As we will explain more precisely in  \Cref{ss:matroids}, the combinatorial data of the family of linearly dependent subsets of $E$ 
is encoded in the arrangement's {\em matroid}. Thus, item (2) above can be rephrased by saying that the integral ring structure depends only on the matroid or equivalently, using a basic fact in matroid theory, that it depends only on the partially ordered 
set 
\begin{equation}
	\mathcal L(\mathcal A) :=\{\cap \mathcal B \mid \mathcal B \subseteq \mathcal A \} \tag{*}
\end{equation}
of all intersections of hyperplanes, ordered by reverse inclusion \cite[\S 2.1]{OrlikTeraoBook}.




The construction of $\Lambda_E/J_E$ can be formally carried out for every abstract matroid, hence with every matroid is associated an \emph{Orlik-Solomon algebra}, and this class of algebras enjoys a rich structure theory (see \cite{YuzvinskiSurvey2001} for a survey).
For instance, the matroid's Whitney numbers of the first kind count the dimensions of the algebra's graded pieces (hence, in the case of arrangements, the Betti numbers of the complement), and generating functions for these numbers can be obtained from classical polynomial invariants of matroids (e.g., the Tutte polynomial).



\subsection{Toric arrangements}\label{ss:TA-TOP}

A \emph{toric arrangement} is a finite  set $\mathcal A$ of 
codimension one 
subtori
in a complex torus $T\simeq (\mathbb{C^*})^n$.
The topological object of interest is, again, the complement $M(\mathcal A):=T\setminus \cup \mathcal A$.
%
%
Each 
such 
subtorus 
can be defined as a coset of the kernel of some character of $T$. 
The arrangement is called \emph{central} if every subtorus is the kernel of a certain character.
If we fix one such defining character for every subtorus in $\mathcal A$ we can consider the matroid of linear dependencies among the resulting set of characters (e.g., viewed as a family of elements of the vector space obtained by tensoring the lattice of characters by $\mathbb Q$). This matroid does not depend on the choice of the characters.

Even to encode basic topological data such as the Betti numbers of the arrangement's complement, this ``algebraic'' matroid data must be refined, for instance by some
``arithmetic'' data given by the multiplicity function 
which keeps track of
the index of sublattices spanned by subsets of the characters.
This approach goes back to Lawrence \cite{Lawrence}. An axiomatization of some crucial properties of this function is the foundation of the theory of {\em arithmetic matroids} \cite{DM,BM}. By \cite{DeConciniProcesiToricArr} and via Moci's arithmetic Tutte polynomial \cite{MociTutte2012}, the Betti numbers of the complement of a 
central
toric arrangement can be computed from the associated arithmetic matroid.

Since intersections of subtori can be topologically disconnected, the ``geometric" intersection data of a toric arrangement is customarily taken to be the poset of {\em layers}, i.e., connected components of intersections (see \Cref{def:posetoflayers}). The significance of this poset was already pointed out by Zaslavsky \cite{Zaslavsky1977}. 
Using the results in this paper Pagaria in \cite{pagaria2} has shown that the integral cohomology algebra of the complement of a toric arrangement is not determined by the poset of layers and that the rational cohomology algebra is not determine by the arithmetic matroid (however it is determined by the poset of layers).
For an in-depth discussion of this question see also \cite{Pagaria2017}. 
The paper \cite{DelucchiRiedel} 
introduces
{\em group actions on semimatroids} 
as an attempt for a unified axiomatization of posets of layers and multiplicity functions.

\medskip

The line of research leading to the present work starts with \cite{DeConciniProcesiToricArr} where a general result about the Betti numbers of the complement 
was obtained (see \Cref{lm:poincare}). Combinatorial models for the homotopy type of complements of toric arrangements were studied in \cite{MociSettepanella,dAntonioDelucchi1}, and minimality of such spaces was proved in \cite{dAntonioDelucchiJEMS}. De Concini and Gaiffi recently computed the cohomology of certain compactifications of $M(\mathcal A)$ \cite{DeCG1,DeCG2}, see also \cite{MociWM} for related earlier work.

%
Algebraic (rational) models for the associated graded of the rational cohomology
of $M(\mathcal A)$ were developed by Bibby \cite{Bibby2016} and Dupont \cite{DupontHypersurfaces2015}, and the minimality result of \cite{dAntonioDelucchiJEMS} implies torsion-freeness of the integral cohomology. Dupont also proved rational formality of $M(\mathcal A)$ in \cite{DupontFormality2016}. Further related work includes results about representation stability \cite{Bibby2} and local system cohomology \cite{DSY}.

Presentations of the graded rational algebra were discussed in \cite{Bibby2016} and further described in \cite{Pagaria2017}, where the dependency of these algebras from the combinatorial data of the poset of layers has been investigated in depth.

The integral cohomology algebra was considered in \cite{CallegaroDelucchi2017} using purely combinatorial methods, but we point out that the formulas for the multiplication given there contain a mistake (see \cite{CallegaroDelucchiErratum}). Here we take a different point of view. 
In particular, we obtain 
a presentation for the cohomology ring $H^*(M(\mathcal A),\mathbb C)$ that can be seen as generalizing the one obtained for hyperplanes by Orlik-Solomon. In the {\em unimodular} case, i.e., when all the intersections of hypertori are connected,
we recover the presentation that had been obtained in \cite{DeConciniProcesiToricArr}.

\subsection{Results} In this paper we provide Orlik-Solomon type presentations for the integral cohomology algebra of a general 
toric arrangement. 

More precisely, 
\begin{itemize}
	\item We generalize De Concini and Procesi's presentation beyond the unimodular case, to 
	all
	toric arrangements (\Cref{thm:main_rational}).
	In the general case this algebra is not necessarily generated in degree one, and every minimal linear dependency among characters induces a number of relations equal to the number of connected components of the intersection of the involved characters (the case where every such dependency induces one relation is precisely the unimodular one studied by De Concini and Procesi). 
	
	\item We prove that the forms we choose as generators of the cohomology are integral. 
	Moreover the relations involved in our presentation hold as relation of forms, not only of cohomology classes.
	Thereby we extend Dupont's result of rational formality  to integral formality, and we obtain an Orlik-Solomon type presentation for the integral cohomology algebra as well (\Cref{thm:main_integral}).
	\item The data needed for the presentation of the rational cohomology is fully encoded in the poset of connected components of intersections (\Cref{combdata}) and, thus, in the $G$-semimatroid associated to the arrangement \cite{DelucchiRiedel}. Moreover, an example due to Pagaria shows that the cohomology ring structure cannot be recovered from the associated arithmetic matroid or from the associated matroid over $\Z$ (see \Cref{rem:matroid_vs_poset} and \cite{pagaria2}).  
\end{itemize}

\subsection{Plan}
The plan of the paper is as follows: First, in \Cref{s:first} we recall a few definitions related to the topology and combinatorics of toric arrangements, and introduce our choice of logarithmic forms associated with the arrangement's elements.
In \Cref{s:some_formal_identities} we start from De Concini and Procesi's work and deduce some formal identities associated with minimal dependencies among the arrangement's defining characters. The technical tool towards treating the non-unimodular case are certain coverings of toric arrangements introduced in \Cref{s:coverings}. Then, in \Cref{s:separation} we put this tool to work and single out a special class of coverings (which we call ``separating covers''). These coverings allow us to define some fundamental forms accounting for the single contributions in cohomology associated with different components of the same intersection.
In \Cref{s:rational_cohomology} we prove that these forms generate the cohomology algebra and the relations generate the whole relation ideal. Finally, in \Cref{sec:integer} we extend our results to integral homology.

\subsection{Acknowledgements} Emanuele Delucchi was supported by the Swiss National Science Foundation professorship grant PP00P2\_150552/1.
Filippo Callegaro and Luca Migliorini were supported by PRIN 2015 ``Moduli spaces and Lie theory'' 2015ZWST2C - PE1.

\section{Basic definitions and notations} \label{s:first}

\subsection{Generalities}
Throughout, $E$ will denote a finite set. For indexing purposes, we will fix an arbitrary total ordering $<$ of $E$ (e.g., by identifying it with a subset of $\mathbb N$). We will also follow the following conventions: we will consider every subset of $E$  to be ordered with the induced ordering. For $A,B\subseteq E$, we will write $(A,B)$ for the concatenation of the two totally ordered sets, \ie if $A=\{a_i<\cdots < a_l\}$ and $B=\{b_i<\cdots < b_h\}$, then $(A,B)=(a_1,a_2,\ldots,a_l,b_1,\ldots,b_h)$, which is typically different from $A \cup B$.

\begin{definition} \label{def:ell}
	Given $A,B\subseteq E$, let $\ell({A,B})$ denote the length of the permutation that takes $(A,B)$ into $A\cup B$.
\end{definition}

\subsection{Toric arrangements}\label{sec:TA}

Let $T=(\mathbb{C}^*)^d$ be a complex torus, and let $\Lambda$ be the lattice of characters of $T$. Consider a list $\underline{\chi}\in \Lambda^{|E|}$ of primitive elements of $\Lambda \simeq H^1(T,\Z)$ and a tuple $\underline{b}\in (\mathbb C^*)^{|E|}$. The \emph{toric arrangement} defined by $\underline{\chi}$ and $\underline{b}$ is 
\[\mathcal{A}=\{H_i\mid i\in E\},\] where $H_i=\chi_i^{-1}(b_i)$ is the level set of $\chi_i$ at level $b_i$, for all $i\in E$.

The toric arrangement is called {\em central} if $\underline{b}=(1,\ldots,1)$, i.e., if $H_i$ is the kernel of $\chi_i$ for all $i\in E$.

\begin{remark}\label{rem:matrix}
	Once an isomorphism of $\Lambda$ with $\mathbb Z^d$ is fixed, for every subset $A\subseteq X$ we can associate the integer $d\times \vert A \vert$ - matrix $[A]$ whose columns are the characters in $A$, say in the fixed ordering of $E$.
\end{remark}

\begin{definition}\label{def:complementaretorico}
	We define $M(\mathcal{A})\subset T$ to be the complement of the toric arrangement $\mathcal{A}$, \ie
	\[
	M(\mathcal{A}):=T\setminus \bigcup_{H \in \A}H.
	\]
\end{definition}

\begin{definition}
	The toric arrangement $\A$ is called \textit{unimodular} if $\cap_{i\in A} H_i$ is either connected or empty for all $A\subseteq E$.
\end{definition}

\begin{definition}\label{def:posetoflayers}
	For a given arrangement $\A$ in a torus $T$ we define the \emph{poset of layers}  $\poset(\A)$ as the set of all connected components of nonempty intersections of elements of $\A$ ordered by reverse inclusion. 
	The elements of $\poset(\A)$ are called \emph{layers} of the arrangement $\A$.
\end{definition}
Notice that the torus $T$  is an element of $\poset(\A)$ since it is the intersection of the empty family of hypertori.
\begin{definition}\label{def:essential}
	The toric arrangement $\A$ is called \emph{essential} if the maximal elements in $\poset(\A)$ are points.
\end{definition}
\begin{remark}\label{rmk:essential}
If $\A$ is 
not essential, all maximal layers in $\poset(\A)$ are translates of the same torus subgroup $W$ of $T$. 
This follows from the classical theory of hyperplane arrangements by applying \cite[Lemma 5.30]{OrlikTeraoBook} to the lifting of $\A$ in the universal covering of $T$.
By choosing any direct summand $T'$ of  $W$ in $T$ we can decompose the ambient torus as $T = W \times T'$. Hence, if we call $\A' = \{H \cap T' \mid H \in \A\}$ the arrangement induced by $\A$ in $T'$, we have that $\A'$ is essential and $M(\A') = M(\A)/W$. Moreover $M(\A) = W \times M(\A')$.
\end{remark}
\begin{definition}\label{def:localarrangement}
	Given a toric arrangement $\A$ in $T$ and a point $p \in T$ we define the linear arrangement $\A[p]$ in the tangent space $\operatorname{T}_p(T)$ as the arrangement given by the hyperplanes $\operatorname{T}_p(H)$ for all $H \in \A$ such that $p \in H$ (cp.\ \S\ref{ss:HA-OS}, and see \cite{OrlikTeraoBook} for background on hyperplane arrangements).
	
	For a given layer $W$ of $\A$, a point $p \in W$ is generic if for any $H \in \A$ such that $W \not \subseteq H$ we have that $p \notin H$.
	We define the linear arrangement $\A[W]$ as the hyperplane arrangement $\A[p]$ for a generic point $p \in W$. 
\end{definition}
\begin{remark}
	Notice that the arrangement  $\A[W]$ does not depend on the choice of the generic point $p$.
\end{remark}

\begin{example} \label{ex:primo}
	Let $x,y$ be the coordinates on the $2$-dimensional torus $T$. We consider the arrangement $\exA$ in $T=(\C^*)^2$ given by the following hypertori:
	\begin{align*}
		H_0:=& \{x^3y=1\};\\
		H_1:=& \{y=1 \};\\
		H_2:=& \{x=1 \}.
	\end{align*}
	Notice that $H_1$ and $H_2$ as well as $H_2$ and $H_0$ intersect in a single point $p= (1,1)$, while $H_1$ and $H_0$ intersect in three points: $p, q= (e^{\frac{2 \pi i}{3}},1), r= (e^{\frac{4 \pi i}{3}},1).$
	
	We can identify the group of characters $\Lambda$ with $\Z^2$ generated by $\chi_1= (0,1)$, $\chi_2=(1,0)$. Hence $y = e^{\chi_1}, x = e^{\chi_2}$ and the hypertorus $H_0$ is associated with the character $\chi_0 = \chi_1 + 3 \chi_2$. 
	
	The intersection of $\exA$ with the compact torus is represented in Figure \ref{fig:esempio}.
	\begin{figure}\label{fig:esempio}
		\begin{tikzpicture}[scale=0.5]
		\draw[gray] (0, 0) rectangle (6,6);
		\draw [line width=1.5pt] (0,3) -- node[below, shift={(0.5,0.05)}]{$H_1$} (6,3);
		\draw [line width=1.5pt] (1,0) -- node[left, shift={(0.1,1)}]{$H_2$} (1,6);
		\draw [line width=1.5pt](0,0) -- (2,6);
		\draw [line width=1.5pt](2,0) -- node[right, shift={(-0.35,1)} ]{$H_0$} (4,6);
		\draw [line width=1.5pt](4,0) -- (6,6);
		\draw (1,3) node [circle, draw, fill=black, inner sep=0pt, minimum width=4pt,label={[shift={(-0.2,0)}]$p$}] {};
		\draw (3,3) node [circle, draw, fill=black, inner sep=0pt, minimum width=4pt,label={[shift={(-0.2,0)}]$q$}] {};
		\draw (5,3) node [circle, draw, fill=black, inner sep=0pt, minimum width=4pt,label={[shift={(-0.2,0)}]$r$}] {};
		\end{tikzpicture}
		\caption{A picture of the arrangement $\exA$.}
	\end{figure}
	Along this paper we will use this arrangement as a running example for the definitions and results that we introduce.
	
	We identify the tangent space $\operatorname{T}_p(T)$ with $\C^2$, with coordinates $\bar x, \bar y$. The local arrangement $\exA[p]$ is given by the hyperplanes with equations $3\bar x+\bar y=0$, $\bar y=0$, $\bar x=0$, while the local arrangement $\exA[q]$ has equations $3\bar x+\bar y=0$, $\bar y=0$.
\end{example}


\subsection{Matroids}\label{ss:matroids} As elements of the vector space $\mathbb Q\otimes \Lambda$, the characters $\chi_i$ determine linear dependency relations. The family
\[
\mathcal C :=\min_{\subseteq}\{C\subseteq E \mid \{\chi_i\}_{i\in C} \textrm{ is a linearly dependent set}\}
\]
of index sets of minimal linear dependencies among characters in $X$ is the set of \emph{circuits} of a {\em matroid} $\mathcal M$ on the set $E$. In the following we will not need specifics about abstract matroids, hence we point to \cite{OxleyBook} for an introduction to this theory.  In general, we will speak of linearly dependent or independent subsets of $E$ referring to dependencies of the corresponding characters.

\def\nbc{\operatorname{nbc}}

\begin{definition}\label{df:nbc}
	Recall the fixed total ordering of $E$. A {\em broken circuit} of the matroid $\mathcal M$ is any subset of $E$ of the form $C\setminus \min{C}$ where $C$ is a circuit.
	
	A {\em no-broken-circuit set} (or {\em nbc}-set) is any subset of $E$ that does not contain any broken circuit. The collection of all nbc sets is denoted $\nbc (\A)$ (or $\nbc(\mathcal M)$ if we want to stress the dependency from the matroid). 
\end{definition}

\begin{remark}
	
	Every nbc-set is necessarily independent.
	
\end{remark}

\subsection{Arithmetic matroids}\label{sec:AM} There is additional enumerative data to be garnered from the set $E$, when this is viewed as a subset of the lattice $\Lambda$. In particular, to every subset $A\subseteq E$ we can associate its {\em span} $\Lambda_A:=\langle A \rangle \subseteq \Lambda$ and a lattice $\Lambda^A:= (\mathbb Q\otimes_{\mathbb Z} \Lambda_A) \cap \Lambda$. 

The function
\[
m: 2^E \to \mathbb N, \quad A\mapsto [\Lambda^A  : \Lambda_A]
\]
that associates to every subset $A$ of $E$ the index of the $\Lambda_A$ in $\Lambda^A$ is called the \emph{multiplicity} of $A$. Equivalently, $m(A)$ is the cardinality of the torsion subgroup of the quotient $\Lambda/\Lambda_A$.


\begin{remark}\label{rem:multiplicity}$\,$
	\begin{enumerate}[label=(\alph*)]
		\item
		If $\A$ is a toric arrangement, then for all $A\subseteq E$ the integer $m(A)$ is the number of connected components of the intersection
		$\bigcap_{\chi_i \in A} H_i$ when this intersection is non-empty (\cf \cite[Lemma~5.4]{MociTutte2012} and \cite[Section 2]{Lawrence2}).
\item Unimodularity of the list $E$ is equivalent to $m$ being constant equal to $1$, and is equivalent to unimodularity of the arrangement $\mathcal A$.
		\item\label{qui} Given a matrix representation as in Remark \ref{rem:matrix}, the number $m(A)$ equals the product of the elementary divisors of $[A]$, i.e., the greatest common divisor of all its minors with size equal to the rank of $[A]$ (\cf 
\cite[\S 16.B]{CurtisReiner}). If $[A]$ is a non-singular square matrix, then $m(A)=\vert \det [A] \vert$.
	\end{enumerate}
\end{remark}

\begin{lemma}\label{lem:arcirc}
	If $C$ is a circuit, then the following relation holds:
	\begin{equation}\label{eq:circuito}
		\sum_{i \in C} {c}_i m(C \setminus \{i\}) \chi_i = 0
	\end{equation}
	where ${c}_i\in \{1,-1\}$ for all $i$.
\end{lemma}

\begin{proof} By definition, as $C$ is minimally dependent, there are nonzero integers $k_i$ such that $\sum_{i\in C} k_i\chi_i = 0$. Moreover, in order to compute the values of the function $m$ it is enough to restrict to the sublattice $\Lambda^C$. By linearity of the determinant function, we have
	\[
	k_i\det [C \setminus \{j\}] = (-1)^{i+j} k_j \det [C \setminus \{i\}] \quad\textrm{ for all } i,j\in C.
	\]
	Now, elementary manipulation shows that \eqref{eq:circuito} holds, e.g., with $c_i=\operatorname{sgn}(k_i)$.

	
\end{proof}

\begin{remark}\label{rm:cipl}
	By \cite[Theorem 3.12]{Pagaria2017}, the coefficients of all relations \eqref{eq:circuito} (and, in particular, the signs $c_i$) depend only on the poset of layers $\mathcal L(\mathscr A)$.
\end{remark}

\begin{remark}\label{rm:unim_pm1}
	If the arrangement is unimodular, from Lemma \ref{lem:arcirc} we garner that every circuit can be realized by a minimal linear dependency all whose coefficients are integer units. 
\end{remark}

The following Lemma is essentially proved in \cite[Thm.~4.2]{DeConciniProcesiToricArr}.

\begin{lemma}\label{lm:poincare}
	If $\A$ is any toric arrangement in a torus $T$ of dimension $d$, the Poincar\'e polynomial of the complement $M(\A)$ is given in terms of the nbc-sets and the multiplicity function as
	\[
	\operatorname{Poin}(M(\A),t)=\sum_{j=0}^d N_j (t+1)^{d-j}t^j,
	\]
	where, for $j=0,\ldots, d$,
	\[
	N_j:= \sum_{L\in \mathcal C_j} \vert \nbc_j(\A[L])\vert.
	\]
	and $\nbc_j(\A[L])$ is the set of no-broken-circuits of cardinality $j$ in the arrangement $\A[L]$.
	In particular, the $j$-th Betti number of $M(\A)$ is
	\[
	\beta_j(M(\A)) = \sum_{i=0}^j N_i \binom{d-i}{j-i}.
	\]
\end{lemma}

\begin{remark}\label{rem:matroid_vs_poset}
	$\,$
	\begin{itemize}
		\item[(a)]
		The data given by the matroid $\mathcal M$ together with the function $m$ determines an {\em arithmetic matroid}. We refer to \cite{DM} for a general abstract definition of an arithmetic matroid, and some of its properties.
		\item[(b)] The poset $\mathcal L(\mathcal A)$ determines the arithmetic matroid data. In fact, for any given set $A\subseteq E$ we can consider the set $X$ of minimal upper-bounds in $\mathcal L(\mathcal A)$: $A$ is independent if and only if the poset-rank of the elements of $X$ equals $\vert A\vert$, and the multiplicity of $A$ equals $\vert X\vert$ (via \Cref{rem:multiplicity}).
		On the other hand the recent paper \cite{pagaria2} explicitly constructs two toric arrangements with isomorphic  arithmetic matroid data but non-isomorphic posets of layers. 
	\end{itemize}
\end{remark}

\begin{example}
	In the arrangement $\exA$ introduced in Example \ref{ex:primo} the only minimal dependent set of characters is $C=\{\chi_0, \chi_1, \chi_2\}$, hence this is the only circuit in the associated matroid. 
	The relation $ -\chi_0 +\chi_1 + 3 \chi_2  = 0$ holds.
	The arithmetic matroid associated with $\exA$ has set $E = \{\chi_0, \chi_1, \chi_2\} \subset \Lambda = \Z^2$ and the multiplicity function is given by
	\[
	m(\{\chi_0, \chi_1\}) = 3,
	\]
	while $m(A)= 1$ for all other subsets of $E$. In particular notice that $\exA$ is a central, not unimodular arrangement.
\end{example}

\subsection{Logarithmic forms} \label{sec:log_forms}
We will study presentations of the cohomology algebra that use, as generators, a distinguished set of logarithmic forms.

\begin{definition} \label{def:omega_oomega_psi}
	For all $i\in E$ we set
	\begin{equation}\label{df:ompsi}
		\omega_i:= \frac{1}{2 \pi \sqrt{-1}}\dlog(1-e^{\chi_i}),\quad \text{ and } \quad
		\psi_i:= \frac{1}{2 \pi \sqrt{-1}} \dlog(e^{\chi_i}).
	\end{equation}
	For symmetry reasons, we also define the forms
	\begin{equation}\label{df:bar}
		\oomega_i:=\frac{1}{2 \pi \sqrt{-1}}\dd\log(1-e^{\chi_i}) + \frac{1}{2 \pi \sqrt{-1}}\dd \log (1-e^{-\chi_i}) = 2 \omega_i -\psi_i.\end{equation}
	Given any $A=\{a_1< \ldots < a_l\}\subseteq E$ we write
	\[
	\psi_A:=\psi_{a_1}\wedge \ldots \wedge \psi_{a_l}.
	\]
	and
	\[
	\omega_A:=\omega_{a_1}\wedge \ldots \wedge \omega_{a_l}, \quad
	\textrm{ resp. } 
	\quad
	\oomega_A:=\oomega_{a_1}\wedge \ldots \wedge \oomega_{a_l}.
	\]
\end{definition}

Now, if $C = \{ \chi_0, \ldots, \chi_k \}\subseteq E$ is a circuit of a unimodular arrangement, and we assume that
\[
\chi_0 = \sum_{i=1}^{k} \chi_1,
\]
De Concini-Procesi in \cite[p.~410, eq.~(20)]{DeConciniProcesiToricArr} (see also Remark \ref{rm:dcpm} below)
prove the formal relation
\begin{equation}\label{eq:20dps}
	\partial \omega_C = 
	\sum_{\substack{\min C \in A\subseteq C,\\ B(A) \neq \emptyset}}  
	(-1)^{\epsilon(A)}\omega_{A}\psi_{B(A)}
\end{equation}
where the fixed total ordering on $E$ is understood,
\begin{align*}
	i(A) & :=\max(C\setminus A),\\ 
	B(A) & :=(C\setminus A) \setminus i(A),\\ 
	\epsilon(A) & :=\vert A\vert + 
	\ell(A,C \setminus A),
\end{align*}
and $\ell(A,C \setminus A)$ is the length of the permutation reordering $A,C \setminus A$ (see \Cref{def:ell}).

\begin{remark}\label{rm:dcpm}
	Notice that in \cite[eq.~(15.3)]{DeConciniProcesiBook} (and also in  \cite[eq.~(20)]{DeConciniProcesiToricArr}) there is a misprint concerning the sign: 
	writing $[k]$ for $\{1, \ldots, k \}$, 
	the correct equation is
	\begin{equation}\label{eq:dp_corretta}
		\omega_{[k]} = \sum_{I \subsetneq [k]}  
		(-1)^{
			|I| + k + 1 + 
			\ell(I, [k] \setminus I )
		}
		\omega_{I}\psi_{B(I \cup \{0\})}\omega_0.
	\end{equation}
\end{remark}

To go from \cite[eq.~(15.3)]{DeConciniProcesiBook} to our \eqref{eq:20dps} it is enough to use the boundary relation
\begin{equation*}
	\partial \omega_C = \omega_{[k]} +
	\sum_{\substack{0\in A\subseteq C,\\ |A|=k}}  
	(-1)^{\epsilon(A)}\omega_{A} .
\end{equation*}

\begin{example} \label{ex:unimod}
	Consider the unimodular arrangement $\exA'$ in $T=(\C^*)^2$ given by the hypertori $ H_1, H_2, H_0$, where $H_0=\{xy=1\}$. 
	The relation
	$ \chi_0 =\chi_1 + \chi_2 $
	holds and the forms associated with $\exA'$ are
	\begin{align*}
		\omega_0=& \frac{1}{2 \pi \sqrt{-1}}\dlog(1-xy),& \omega_1=& \frac{1}{2 \pi \sqrt{-1}}\dlog(1-y),& \omega_2=& \frac{1}{2 \pi \sqrt{-1}}\dlog(1-x),\\
		\psi_0= &\frac{1}{2 \pi \sqrt{-1}} \dlog(xy),&
		\psi_1= & \frac{1}{2 \pi \sqrt{-1}} \dlog(y),&  \psi_2= & \frac{1}{2 \pi \sqrt{-1}} \dlog(x).
	\end{align*}
	The set $C = \{\chi_0,\chi_1, \chi_2\}$ is the only circuit and relation \eqref{eq:20dps} gives
	\[
	\omega_0 \omega_1 - \omega_0 \omega_2 + \omega_1 \omega_2 = \omega_0 \psi_1
	\]
	as can be checked directly. In the following we will use the arrangement $\exA'$ as a running example of a unimodular arrangement.
\end{example}

%

\section{Some formal identities}\label{s:some_formal_identities}

%

In this section we derive some identities among the forms associated with a circuit $C\subseteq E$. For ease of notation we identify $E$ as a subset of $\mathbb N$ with the natural order, and we suppose that $C=\{0,1,\ldots,k\}$. Then the characters $\chi_0,\ldots,\chi_k$ exhibit a linear dependency, and we  examine different cases according to the signs of the coefficients of this linear dependency. 

The results of this section will be enough in order to treat the unimodular case, where (see Remark~\ref{rm:unim_pm1}) such coefficients must be units. 

\begin{lemma}\label{lem:prodotto}
	If $\chi_0 = \sum_{i=1}^k  \chi_i $, we have the following  identity.
	\begin{equation}\label{eq:prodotto}
		\omega_1\cdots \omega_k
		=
		\omega_0 \prod_{i=2}^k (\omega_i - \omega_{i-1} + \psi_{i-1})\end{equation}
\end{lemma}
\begin{proof}
	We fix $j \in \{1, \ldots, k\}$. Consider the non-zero products in the expansion of \eqref{eq:prodotto} that do not contain either the factor $\omega_j$ nor the factor $\psi_j$. In each one of these terms, all the factors $\omega_i$ for $i>j$ have to appear.
	Instead, due to the fact that $\omega_i\wedge  \psi_i = 0$,  
	exactly one of the two  terms $\omega_i$ and $\psi_i$
	has to appear for $i<j$ . So 
	the sum of the products not containing $\omega_j$ or $\psi_j$ will be
	\[
	\omega_0\prod_{1 \leq i<j} (-\omega_i+\psi_i) \prod_{i>j} \omega_i.
	\]
	Hence we have,
	\begin{align*}
		\omega_0 \prod_{i=2}^k (\omega_i - \omega_{i-1} + \psi_{i-1}) & =  \sum_{j=1}^k \omega_0\prod_{1 \leq i<j} (-\omega_i+\psi_i) \prod_{i>j} \omega_i = \\
		& =\sum_{j=1}^k \sum_{\substack{0\in A\subsetneq C,\\ i(A) = j}} 
		(-1)^{|A_{\leq j}|-1} \eta_A
	\end{align*}
where 
	\[
	\eta_A = \eta_0 \cdots \widehat{\eta_{i(A)}} \cdots \eta_k
	\] 
	and
	\[
	\eta_i := \left\{
	\begin{array}{ll} \omega_i & \text{ if }i\in A\\ 
	\psi_i & \text{ otherwise.}
	\end{array}
	\right. 
	\]  
	We conclude the purely formal identity
	\begin{equation}\label{eq:pf}
		\omega_0 \prod_{i=2}^k (\omega_i - \omega_{i-1} + \psi_{i-1}) =  \sum_{\substack{0\in A\subsetneq C}} 
		(-1)^{|A_{\leq i(A)}|-1} \eta_A.
	\end{equation}
	Now we use our assumption $\sum_{i=1}^k  \chi_i = \chi_0$. It entails that the form $\vartheta^{(0)}$ defined before Proposition 15.6 in \cite{DeConciniProcesiBook} equals $\omega_0$. In particular, again in the notation of \cite{DeConciniProcesiBook}, for $I \subset [k]$ we have
	\[
	\Phi_I^{(0)}  \stackrel{\textrm{def}}{=} (-1)^{\ell(I, [k] \setminus I)}\prod_{i \in I} \omega_i \prod_{j\in 
		B(I \cup \{0\})
	} \psi_j  \vartheta^{(0)} = (-1)^{i(I)-1}  \eta_{I \cup \{0\}}
	\] 
	noticing that the products $\eta_{I \cup \{0\}}$ already follow the standard ordering. We can now use \cite[eq.~(15.3)]{DeConciniProcesiBook}, i.e.,
	\[
	\sum_{I \subsetneq [k]} (-1)^{|I| + k
		+1
	} 	\Phi_I^{(0)} = \omega_1 \cdots \omega_k
	\]
	in order to rewrite Equation \eqref{eq:pf}. If we take $A = I \cup\{0\}$, since $k-i(A) = |A|-|A_{\leq i(A)}|,$
	we obtain the claimed equality.
\end{proof}
\begin{lemma}\label{lm:allplus}
	If $\sum_{i=0}^k \chi_i = 0$, then we have
	\begin{equation}\label{eq:prod_a}
		\prod_{i=1}^k (\omega_i - \omega_{i-1} + \psi_{i-1}) = 0
	\end{equation}
	
	or, using the forms $\oomega_i$ defined in  Equation~\ref{df:bar},
	\begin{equation}
		\prod_{i=1}^{k} (\oomega_i + \psi_i - \oomega_{i-1} + \psi_{i-1})=0.
	\end{equation}
	
\end{lemma}
\begin{proof}
	We start by a formal identity which can be readily verified, e.g., by induction on $k$.
	\[
	\omega_1 \cdots \omega_k = \omega_1 \prod_{i=2}^{k} (\omega_i - \omega_{i-1} + \psi_{i-1})
	\]
	We can now expand the left-hand side using Lemma \ref{lem:prodotto} applied to the identity $(-\chi_0) = \sum_{i>0} \chi_i$. Collecting terms we obtain
	\[
	0= \left(\omega_1 - \frac{1}{2 \pi \sqrt{-1}} \dd \log (1 - e^{-\chi_0})\right) \prod_{i=2}^{k} (\omega_i - \omega_{i-1} + \psi_{i-1}).\]
	Noticing that $\frac{1}{2 \pi \sqrt{-1}}\dd \log (1 - e^{-\chi_0}) = \omega_0 - \psi_0$ we conclude:
	\[
	(\omega_1- \omega_0 + \psi_0) \prod_{i=2}^{k} (\omega_i - \omega_{i-1} + \psi_{i-1}) = 0.\]

	For the second equation we can immediately compute
	\[2(\omega_i - \omega_{i-1} + \psi_{i-1})=\oomega_i + \psi_i - \oomega_{i-1} + \psi_{i-1},\] so multiplying formula \eqref{eq:prod_a} by $2^k$ we get the claimed identity.
\end{proof}

\begin{example}
	We continue with the arrangement introduced in Example \ref{ex:unimod}.
	Since the relation
	\[
	\chi_0 = \chi_1 + \chi_2
	\]
	holds, from Lemma \ref{lem:prodotto} we have
	$ \omega_1 \omega_2 = \omega_0 (\omega_2-\omega_1 + \psi_1).
	$
	In order to apply Lemma \ref{lm:allplus} we set $\chi_0' := - \chi_0$ (and hence $\omega'_0 = \omega_0-\psi_0$ and $\psi'_0 = -\psi_0$) and if we consider the characters $\chi_0', \chi_1, \chi_2$ and the ordering $(0,1,2)$ for the elements of  the circuit we obtain
	\[
	(\omega_1 - \omega_0' + \psi_0')(\omega_2 - \omega_1 + \psi_1)=0,
	\]
	while if we consider the ordering $(0,2,1)$ we obtain the relation
	\[(\omega_2-\omega_0' + \psi_0')(\omega_1 - \omega_2 + \psi_2) = 0
	\]
	as one can easily check by direct computation. The relations that we can obtain with different orderings of the elements in the circuit are consequences of the two above.
\end{example}

%


\begin{lemma}\label{lm:signedeq}
	If $\sum_{i=0}^k {c}_i \chi_i = 0$ where ${c}_i = \pm 1$ for all $i$, 
	\begin{equation}\label{eq:prod_b}
		\prod_{i=1}^{k} (\oomega_i + {c}_i\psi_i - \oomega_{i-1} + {c}_{i-1}\psi_{i-1})=0,
	\end{equation}
\end{lemma}

\begin{proof} We apply Lemma \ref{lm:allplus} to the identity $\sum_{i=0}^k\chi_i'=0$ where we set $\chi_i':=c_i\chi_i$ for all $i$. A glance at Equations \eqref{df:ompsi} and \eqref{df:bar} shows that the forms $\oomega_i'$ and $\psi_i'$ associated with the $\chi_i'$ satisfy $\oomega_i'=\oomega_i$ and $\psi_i'=c_i\psi_i$ for all $i$, proving the claimed equality.
	%
	%
	%
\end{proof}

\begin{definition}\label{df:etabarra}
	Given a subset $A\subseteq E$, for every $i\in E$ let
	\[
	{\oeta}_i^A:= \left\{
	\begin{array}{ll} \oomega_i & \text{ if }i\in A\\ 
	\psi_i & \text{ otherwise}
	\end{array}
	\right. .
	\]
	Thus, if $B\subseteq E$ is disjoint from $A$ we can define
	\[
	\oeta_{A,B} := \prod_{i \in A \cup B}\oeta_i^A,
	\]
	where the factors are in increasing order with respect to the total order on $E$.
\end{definition}



\begin{proposition}\label{prop:rel_unimod} Let $C$ be a circuit of the matroid such that the corresponding minimal linear dependency has the form $\sum_{i\in C} c_i\chi_i=0$ where $c_i\in \{\pm 1\}$ for all $i$. Then,
	
	\begin{equation}\label{eq:raccolta}
		\sum_{j\in C} \sum_{
			\substack{
				A,B\subset C\\
				C= A \sqcup B \sqcup \{j\}
		}} (-1)^{|A_{\leq j}|} c_{B}\oeta_{A,B}  = 0
	\end{equation}
	where, 
	for every $B\subseteq E$, we write $c_B := \prod_{i \in B} c_i$.
	Moreover, as a consequence of the equation above we have
	\begin{equation}\label{eq:relazione_bar}
		\sum_{j \in C} \sum_{
			\substack{
				A,B\subset C\\
				C= A \sqcup B \sqcup \{j\}\\ \vert B \vert \textrm{ even}
			}
		} (-1)^{|A_{\leq j}|} c_B\oeta_{A,B}  = 0.
	\end{equation}
	In particular $\partial \oomega_C$ corresponds to the sum of the terms with $B = \emptyset$.
\end{proposition}

\begin{proof}
	Equation~\eqref{eq:prod_b} can be rewritten as follows:	
	\begin{equation}
		\sum_{j=0}^k \prod_{i<j} (-\oomega_i + c_i\psi_i) \prod_{i>j} (\oomega_i + c_i\psi_i) = 0
	\end{equation}
	Expanding all the products and using Definition~\ref{df:etabarra} we obtain formula~\eqref{eq:raccolta}.
	
	\label{rm:paridispari}
	Moreover, using the negated equation $\sum_{i=0}^k - {c}_i \chi_i = 0$, Lemma \ref{lm:signedeq} gives
	\begin{equation}\label{eq:negata}
		\prod_{i=1}^{k} (\oomega_i - {c}_i\psi_i - \oomega_{i-1} - {c}_{i-1}\psi_{i-1})=0.
	\end{equation}
	Adding this relation to the one in \eqref{eq:prod_b}, and decomposing the expansion of the product 
	in two parts, one containing all the terms $\omega_A\psi_B$ with $|B|$ even and the other one containing all those terms with $|B|$ odd, it can be shown that each of the two parts must equal $0$.
\end{proof}

In \cite[Thm.~5.2]{DeConciniProcesiToricArr} De Concini and Procesi prove that the complement of a unimodular toric arrangement is formal. They do this by showing that the rational cohomology ring is isomorphic to the sub-algebra of closed forms generated by $\omega_i=\dlog(e^{b_i}-e^{\chi_i})$ for $i \in E$ and $\psi_\chi=\dlog(e^{\chi})$ for $\chi \in \Lambda$. The formal relations among these generators are implicit in \cite[eq.~(20)]{DeConciniProcesiToricArr}.

Notice that if the arrangement $\A$ is essential the forms  $\psi_i=\dlog(e^{\chi_i})$ for $i \in E$ generate $H^1(T; \Q)$.  It follows that the relations stated in \Cref{prop:rel_unimod}  above lead to a presentation of the cohomology ring
with respect to the generators
$\oomega_i$'s and $\psi_i$'s. Hence we have the following reformulation of the result of \cite{DeConciniProcesiToricArr}.
\begin{theorem}
	Let $\A$ be an essential unimodular toric arrangement. The rational cohomology algebra $H^*(M(\A),\mathbb Q)$ is isomorphic to the algebra $\mathcal E$ with 
	\begin{itemize}
		\item[--] Set of generators $e_{A;B}$, where $A$ and $B$ are disjoint and such that $A\sqcup B$ is an independent set;
		the degree of the generator $e_{A;B}$ is $\vert A \sqcup B\vert$.
		\item[--] The following types of relations
		\begin{itemize}
			\item For any two generators $e_{A;B} $, $e_{A';B'}$,  
			\begin{equation}\label{eq:rel_prodotto_nullo_xyz}
			e_{A;B} e_{A';B'}=0
			\end{equation} if $ A \sqcup B \sqcup A' \sqcup B' $  is a dependent set, and otherwise 
			\begin{equation}\label{eq:relazione_prodotto_xyz}
				e_{A;B}e_{A';B'}= 
				(-1)^{
					\ell(A\cup B, A' \cup B')
				} 
				e_{A\cup A';B\cup B'}.
			\end{equation}
			\item For every linear dependency $\sum_{i\in E} n_i \chi_i=0$ with $n_i\in \mathbb Z$, a relation 
			\begin{equation}\label{eq:relazione_toro_xyz}
				\sum_{i\in E} n_i e_{\emptyset;\{i\} } =0.
			\end{equation}
			
			\item  For every circuit $C\subseteq E$, with associated linear dependency $\sum_{i\in C} n_i\chi_i$ with $n_i\in \mathbb Z$, a relation
			\begin{equation}\label{eq:relazione_final_unimod}
				\sum_{j\in C} \sum_{
					\substack{
						A,B\subset C\\
						C= A \sqcup B \sqcup \{j\}\\
						\vert B \vert \textnormal{ even}
				}}
				(-1)^{|A_{\leq j}|} c_B   e_{A;B} = 0
			\end{equation}
			where, for all $i\in C$, $c_i:=\operatorname{sgn}n_i$ and $c_B = \prod_{i \in B} c_i$.			
		\end{itemize}
	\end{itemize}     
\end{theorem}

\begin{remark}
In order to check that the presentation above gives the same algebra described in \cite{DeConciniProcesiToricArr}, we can first notice that relation \eqref{eq:relazione_prodotto_xyz} implies that our algebra is generated in degree $1$, by elements of the form $e_{\{i\}; \emptyset}$ and $e_{\emptyset; \{i\}}$, that correspond respectively to the generators $\lambda_{a_i, \chi_i}$ and $\omega_i$ in \cite[p.~410]{DeConciniProcesiToricArr}.  Then our relation \eqref{eq:rel_prodotto_nullo_xyz} corresponds to relation (2)  of \cite[p.~410]{DeConciniProcesiToricArr}; our relation \eqref{eq:relazione_toro_xyz} corresponds to relation (1)  of \cite[p.~410]{DeConciniProcesiToricArr} and our relation \eqref{eq:relazione_final_unimod} corresponds to relation (20') that is implicit in \cite{DeConciniProcesiToricArr}.
\end{remark}

\begin{example}
	Continuing Example \ref{ex:unimod} and using the relation $-\chi_0+\chi_1+\chi_2=0$ 
	we obtain that the rational cohomology of the complement of arrangement $\exA'$ has a presentation with generators 
	\[\oomega_0, \oomega_1, \oomega_2, \psi_0, \psi_1, \psi_2\]
	where $-\psi_0 + \psi_1 + \psi_2 = 0$ and
	relation \eqref{eq:relazione_bar} (or equivalently relation \eqref{eq:relazione_final_unimod}) gives
	\[
	\oomega_0 \oomega_1 - \oomega_0 \oomega_2 + \oomega_1 \oomega_2 - \psi_0\psi_1 - \psi_0\psi_2 + \psi_1\psi_2=0.
	\]	
	Note that $\psi_0\psi_1 + \psi_0\psi_2 = \psi_0 \psi_0 = 0$ and hence the relation above can be simplified.
\end{example}

\section{Coverings of arrangements}\label{s:coverings}

Recall our setup from Section \ref{sec:TA}, and in particular that we consider a primitive arrangement $\A$ in a torus $T$. 

Given a lattice $\Lambda'$, $\Lambda \subseteq \Lambda' \subseteq \Lambda  \otimes \Q$ we consider the Galois covering $U \to T$ associated with the subgroup $\Lambda'^* \subseteq \Lambda^*  \simeq \pi_1(T)$ whose  group of deck automorphisms is $(\Lambda'/\Lambda)^* \simeq \operatorname{Gal}(U/T)$.

\begin{definition}
	Let $f:U\to T$ be a finite covering, and call $\A_U$ the lift of $\A$ through $f$ to the torus $U$. More precisely, let
	\[
	\A_U:=\bigcup_{H\in \A} \pi_0(f^{-1}(H)),
	\]
	the set of connected components of preimages of hypertori in $\A$.
	Moreover, given $i\in E$
	let
	\[
	a_i:=\vert \pi_0(f^{-1}(H_i)) \vert
	\]
	denote the number of connected components of $f^{-1}(H_i)$. Given $q\in f^{-1}(H_i)$, let 
	\begin{equation*}
		H_i^U(q) 
	\end{equation*} 
	denote the connected component of $f^{-1}(H_i)$ containing $q$.
	
\end{definition}

\begin{remark}\label{rem:charU}
	The previous definition ensures that $\A_U$ is again a primitive arrangement. It is, however, not necessarily central. 
	
	In fact, if we call $\widehat{\chi}:= f\circ \chi$ the character of $U$ induced by $\chi$, we see that the connected components $ f^{-1}(H_i)$ are associated with the (primitive) character $\frac{\widehat{\chi}_i}{a_{i}}$. More precisely, every $L\in \pi_0(f^{-1}(H_i))$ has equation
	\[
	\frac{\widehat{\chi}_i}{a_i} =  \frac{\widehat{\chi}_i}{a_i}(q)
	\]
	where $q$ is any point of $L$.
\end{remark}

\subsection{Logarithmic forms on coverings}

%



Our next task is to describe the logarithmic forms on $M(\A_U)$ associated with $\A_U$ according to the set-up of Section \ref{sec:log_forms}. 

Let $f:U\to T$ be a finite covering and let $M(\A_U)$ be as above.
The algebraic de Rham complex $\Omega_{M(\A_U)}^\bullet$ splits as direct sum of subcomplexes 
\begin{equation}
	\label{eq:decompos_seminvarianti}
	\Omega_{M(\A_U)}^\bullet \simeq \bigoplus_{\lambda \in \Lambda'/\Lambda} \Omega^\bullet_\lambda
\end{equation}
where $\Omega^\bullet_\lambda$ consists of forms $\alpha$ such that for any $\tau \in \operatorname{Gal}(U/T)$ we have that $\tau^*(\alpha) = \lambda(\tau) \alpha$. In particular the subcomplex of invariant forms $\Omega^\bullet_\mathbbm{1}$ is canonically identified with $\Omega^\bullet_{M(\A)}$.

For any $i\in E$ and any point $q\in f^{-1}(H_i)$ we
set
\begin{equation}\label{eq:def_omega_U}
	\omega_i^U(q) :=\frac{1}{2 \pi \sqrt{-1}} \dd \log \left(1 - e^{\frac{\widehat{\chi}_i}{a_i}-\frac{\widehat{\chi}_i}{a_i}(q)}\right)
\end{equation}
for the logarithmic form 
in $\Omega^1_{M(\A_U)}$ associated with $H_i(q)$.  Notice that this form does not depend on the choice of $q$ in the same connected component. 

Moreover, let
\begin{equation}
	\label{df:psiU}
	\psi_i^U:= \frac{f^*(\psi_i)}{a_i}
	= \dd\log e^{\frac{\widehat{\chi}_i}{a_i}
	},
\end{equation}
where the upper symbol $^*$ denotes as usual the pull-back.


More generally, given any $A\subseteq E$, choose $q\in f^{-1}(\cap_{i\in A} H_i)$ and 
let
%
\begin{equation}\label{eq:prodotti_omega}
	\oomega_A^U(q) := \prod_{i \in A} \oomega_i^U(q) 
	\quad
	\textrm{ and }
	\quad
	\omega_A^U(q) := \prod_{i \in A} \omega_i^U(q) ,
\end{equation}
where the factors are taken in increasing order with respect to the index $i$ and $\oomega_i^U(q) := 2 \omega^U_i(q) - \psi^U_i$.
%
Moreover, under the same set-up, for disjoint $A$ and $B$ let
\[
\oeta^U_{A,B}(q) := \prod_{i \in A \sqcup B}\oeta^U_i(q)
\]
where		$\oeta_i^U(q) = \oomega_i^U(q)$ if $i \in A$ and $\oeta_i^U(q) = \psi_i^U$ if $i \in B$.
\begin{proposition}
	Let $A$ be any set of indices and let $W$ be a connected component of $\bigcap_{i\in A} H_i$. Let $p$ be any point in $W$.
	The class
	\[
	\sum_{q \in f^{-1}(p)} \oomega_A^U(q)
	\]
	is invariant with respect to the group $G$ of deck automorphisms of $f$ and it does not depend on the choice of the point $p$ in $W$.
\end{proposition}
\begin{proof} The only nontrivial case is when the characters associated with the indices in $A$ are linearly independent, otherwise $\oomega^U_A(q)=0$.
	
	Let $\tau \in G $. Using the definitions we have the equalities
	\begin{align*}
		\tau^*({\omega}^U_i(q)) & = \tau^*\left(\frac{1}{2 \pi \sqrt{-1}} \dd\log\left(1-e^{\frac{\chi_i}{a_i}-\frac{\chi_i}{a_i}(q)}\right) \right)  \\
		& =  \frac{1}{2 \pi \sqrt{-1}}\dd\log\left(1-e^{\frac{\chi_i}{a_i} + \frac{\chi_i}{a_i}(\tau) -\frac{\chi_i}{a_i}(q)}\right) \\
		&= \frac{1}{2 \pi \sqrt{-1}}\dd\log\left(1-e^{\frac{\chi_i}{a_i} -\frac{\chi_i}{a_i}(\tau^{-1}q)}\right)\\
		& = {\omega}^U_i(\tau^{-1}(q)).
	\end{align*}	
	Since the forms $\psi_i$ are translation-invariant, we obtain immediately also
	\[
	\tau^*({\oomega}^U_i(q))= {\oomega}^U_i(\tau^{-1}(q)).
	\]
	If we write $A=\{a_1,\ldots,a_k\}$, we see that every form 
	\[
	{\oomega}^U_A(q) = {\oomega}^U_{a_1}(q) {\oomega}^U_{a_1}(q)\cdots {\oomega}^U_{a_k}(q)
	\] 
	satisfies 
	\[
	\tau^* \left({\oomega}^U_A(q) \right)= {\oomega}^U_A(\tau^{-1}(q)).
	\]
	The claim follows.
\end{proof}

The previous result allows us to give the following definition.

\begin{definition} \label{def:oomegaf}
	Let $\A=\{H_i\}_{i\in E}$ be a toric arrangement in the torus $T$ and consider a finite covering $f:U\to T$.  Consider an independent set $A\subseteq E$, let $W$ be a connected component of $\cap_{i\in A} H_i$ and choose $p\in W$.
	Since the pullback map $f^*$ is injective, we can define forms $\oomega_{W,A}^f$ and $\omega_{W,A}^f$ as the unique forms on $M(\A)$ such that
	\[
	f^*(\oomega_{W,A}^f) = \frac{1}{|\cap_{i \in A} H_i^U(q_0)\cap f^{-1}(p)|} \sum_{q \in f^{-1}(p)} \oomega_A^U(q)
	\]
	and
	\[
	f^*(\omega_{W,A}^f) = \frac{1}{|\cap_{i \in A} H_i^U(q_0)\cap f^{-1}(p)|} \sum_{q \in f^{-1}(p)} \omega_A^U(q)
	\]
	where $q_0$ is any point in $f^{-1}(p)$. 
\end{definition}

\begin{remark} \label{rem:fstaromega}
	If the arrangement $\A_U$ is unimodular the formula in the definition above becomes
	\[
	f^*(\oomega_{W,A}^f) = \frac{1}{|L\cap f^{-1}(p)|} \sum_{q \in f^{-1}(p)} \oomega_A^U(q)
	\]
	where $L$ is any connected component of $f^{-1}(W)$.
\end{remark}

\begin{example}\label{ex:cover}
	We can now consider the arrangement $\exA$ of \Cref{ex:primo} and the covering $f:U = (\C^*)^2 \to T=(\C^*)^2$ given by $(u,v) \mapsto (u, v^3)$. The arrangement $\exA_U$ is unimodular and is given by the $7$ hypertori with equations $u=1$, $v= e^{\frac{2 \pi i a}{3}}
	$ and $uv = e^{\frac{2 \pi i b}{3}}$ for 
	$a,b=0,1,2$.
	
	The three subarrangements of $\exA_U$ containing respectively the hypertori passing through $p_1=(1,1), p_2=(1, e^{\frac{2 \pi i }{3}}), p_3=(1, e^{\frac{4 \pi i }{3}})$ are, up to translation, isomorphic to the unimodular arrangement $\exA'$, while the  subarrangements passing through the other six points are all isomorphic to the boolean arrangement in $(\C^*)^2$ given by the hypertori with equations $x=1$ and $y=1$.
	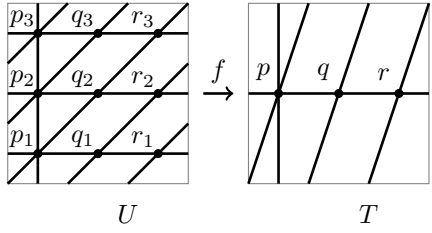
\begin{figure}[htb]
		\begin{tikzpicture}[scale=0.4]
		\draw [gray] (0, 0) rectangle (6,6);
		\node [draw=none,fill=none] at (7-3,3-4) {$U$};
		\draw [line width=1.0pt](1,0) -- (1,6);
		\draw [line width=1pt] (0,1) -- (6,1);
		\draw [line width=1pt] (0,3) -- (6,3);
		\draw [line width=1pt] (0,5) -- (6,5);
		\draw [line width=1pt](0,4) -- (2,6);
		\draw [line width=1pt](0,2) -- (4,6);
		\draw [line width=1pt](0,0) -- (6,6);
		\draw [line width=1pt](2,0) -- (6,4);
		\draw [line width=1pt](4,0) -- (6,2);
		\foreach \x in {1,2,3}{
			\draw (1,2*\x-1) node [circle, draw, fill=black, inner sep=0pt, minimum width=3pt,label={[shift={(-0.2,-0.1)}]$p_\x$}] {};
			\draw (3,2*\x-1) node [circle, draw, fill=black, inner sep=0pt, minimum width=3pt,label={[shift={(-0.2,-0.1)}]$q_\x$}] {};
			\draw (5,2*\x-1) node [circle, draw, fill=black, inner sep=0pt, minimum width=3pt,label={[shift={(-0.2,-0.1)}]$r_\x$}] {};};
		\draw [line width=1pt,->] (3+3.5,3) -- node[above]{$f$} (4+3.5,3);
		\draw [gray] (0+8, 0) rectangle (6+8,6);
		\draw [line width=1pt] (0+8,3) -- (6+8,3);
		\draw [line width=1pt](0+8,0) -- (2+8,6);
		\draw [line width=1pt](2+8,0) -- (4+8,6);
		\draw [line width=1pt](4+8,0) -- (6+8,6);
		\draw (1+8,3) node [circle, draw, fill=black, inner sep=0pt, minimum width=3pt,label={[shift={(-0.2,0)}]$p$}] {};
		\draw (3+8,3) node [circle, draw, fill=black, inner sep=0pt, minimum width=3pt,label={[shift={(-0.2,0)}]$q$}] {};
		\draw (5+8,3) node [circle, draw, fill=black, inner sep=0pt, minimum width=3pt,label={[shift={(-0.2,0)}]$r$}] {};
		\node [draw=none,fill=none] at (7+8-3,3-4) {$T$};
		\draw [line width=1pt](1+8,0) -- (1+8,6);
		\end{tikzpicture}
		\caption{A picture on the compact torus of the covering described in Example \ref{ex:cover}.} \label{fig:covering}
	\end{figure}
	We have the form 
	\begin{align*}
		f^*(\oomega^f_{p,\{0,1\}}) = & \frac{-1}{4\pi^2}\left(
		\frac{1+uv}{1-uv}\frac{\dd(uv)}{uv}  
		\frac{1+v}{1-v}\frac{\dd v}{v} +
		\frac{1+\zeta_3 uv}{1-\zeta_3 uv}\frac{\dd(uv)}{uv}  
		\frac{1+\zeta_3 v}{1-\zeta_3 v}\frac{\dd v}{v} + \right.\\ &+ \left. 
		\frac{1+\zeta_3^2 uv}{1-\zeta_3^2 uv}\frac{\dd(uv)}{uv}  
		\frac{1+\zeta_3^2 v}{1-\zeta_3^2 v}\frac{\dd v}{v} 
		\right) = \\
		= & \frac{-3}{4\pi^2} \frac{u^3 v^6+u^3v^3 +4u^2v^3+4uv^3+v^3 +1}{u v \left(v^3-1\right) \left(u^3 v^3-1\right)} \dd u \dd v
	\end{align*}
	where $\zeta_3 = e^{\frac{2 \pi i}{3}}$ and hence, taking the pushforward and dividing by the degree, we get
	\begin{align*}
		\oomega^f_{p, \{0,1\}} = & \frac{-1}{4\pi^2}\frac{x^3 y^2+x^3y +4x^2y+4xy+y +1}{xy \left(y-1\right) \left(x^3y-1\right)} \dd x \dd y.
	\end{align*}
\end{example}

\section{Separation}\label{s:separation}

To deal with the non-unimodular case, the following definition turns out to be useful.

\begin{definition}\label{def:separates}
	Let $A$ be an independent subset of $E.$ We say that a covering $f:U \to T$ \emph{separates} $A$ if, for any connected component $W$ of $\cap_{i \in A} H_i$ and for all $i \in A$ there exist $q_i\in f^{-1}({H_i})$ such that 
	$f(\cap_{i \in A} H_i^U(q_i)) = W$.
\end{definition}

\begin{remark}
	If $f:U \to T$ is a covering such that the arrangement $\A_U$ is unimodular, then $f$ separates $A$ for all independent set $A \subset E$.
\end{remark}

\begin{proposition}\label{ex:sepind}
	Let $A\subseteq E$ be an independent set. 
	There exists a covering $f:U \to T$ that separates $A$.
\end{proposition}
\begin{proof}
Let $\Gamma$ be a direct summand of $\Lambda^A$ in $\Lambda$. Hence $\Lambda = \Lambda^A \oplus \Gamma$.
	Consider the lattice 
	\[
	\Lambda(A):= \left\langle\frac{\chi_i}{m(A)}\right\rangle_{i \in A }\oplus\Gamma \subseteq  \Lambda \otimes \mathbb Q.
	\]
	We have the tower of subgroups 
	\[
	\Lambda_A \subseteq \Lambda^A \subseteq \left\langle\frac{\chi_i}{m(A)}\right\rangle_{i \in A }
	\]
	where by definition $[\Lambda^A:\Lambda_A]=m(A)$. Moreover since $A$ is independent we have $[\big\langle\frac{\chi_i}{m(A)}\big\rangle_{i \in A }:\Lambda_A] = m(A)^{|A|}$. Hence we have $[\big\langle\frac{\chi_i}{m(A)}\big\rangle_{i \in A }:\Lambda^A] = 
	m(A)^{\vert A \vert -1}$.
	The inclusion $\Lambda \subseteq \Lambda(A)$ induces a covering  $f: U\to T$ of degree 
	$[\Lambda(A): \Lambda] = [\big\langle\frac{\chi_i}{m(A)}\big\rangle_{i \in A }:\Lambda^A] = 
	m(A)^{\vert A \vert -1}$. The first equality follows since $\Gamma$ is a direct summand of both terms in the left hand side. 
	
	This covering separates $A$.
	In fact, we claim that for every connected component $W$ of $\bigcap_{i\in A}H_i$ and any choice of a point $q\in f^{-1}(W)$, the intersection $\bigcap_{i\in A}H_i^U(q)$ is connected. To prove this claim, let $k$ denote the number of connected components of $\bigcap_{i\in A} H_i^U(q)$. We count in two different ways the number of connected components of $f^{-1}(\bigcap_{i\in A} H_i)$. 
	On the one hand, for every $i$ we have $m(A)$ connected components of $f^{-1}(H_i)$ and, once we have chosen for every $i$ a connected hypertorus in $f^{-1}(H_i)$, their intersection has $k$ connected components. In this way we have $k m(A)^{\vert A \vert}$ such components.
	On the other hand, the number of connected components of the preimage of each of the $m(A)$ connected components of $\bigcap_{i\in A} H_i$ is at most $m(A)^{\operatorname{rk}(A) -1}$, the degree of the covering -- hence we obtain a count of at most $m(A)^{\operatorname{rk}(A)} = m(A)^{|A|}$ components. We conclude $k=1$. 
\end{proof}
The following theorem motivates our definition of separating coverings.
\begin{theorem} \label{th:coverings}
	Let $A\subset E$ be an independent set. If $f:U \to T$ and $g:V \to T$ are coverings that separate $A$, then $\oomega_{W,A}^f = \oomega_{W,A}^g$.
	Analogously we have $
	\omega_{W,A}^f = \omega_{W,A}^g
	$.
\end{theorem}
In the proof we will make use of the following remark.
\begin{remark}\label{rem:riparametrizzazione}
	For every index $i$, let $H^U_{i,1},\ldots H^U_{i,m_i}$, denote the connected components of $f^{-1}(H_i)$ and $ \oomega^U_{i,j} := \oomega_{H^U_{i,j}}$ be the associated forms.

	
	If we assume that $f$ separates $A$, then Definition~\ref{def:oomegaf} is equivalent to
	\[
	f^*(\oomega^f_{W,A}) = \sum_{\substack{\uno \leq \bj \leq \bm\\ \cap_i H^U_{i, j_i} \subseteq f^{-1}(W)}} \prod_{i \in A} \oomega^U_{i,j_i}
	\]
	where the sum is indexed using the componentwise ordering among integer $A$-tuples $\uno:=(1, \ldots, 1)$,  $\bj:=(j_i)_{i\in A}$, $\bm := (m_i)_{i\in A}$.
\end{remark}

\begin{proof}[Proof of Theorem~\ref{th:coverings}]
	We give the proof for $\oomega_{W,A}^f = \oomega_{W,A}^g$, the other case being identical.
	
	The theorem follows in its generality if we first assume that the statement holds when $g = f \circ h$, where $h:V\to U$ is a finite covering. In this case  we have
	\[
	\xymatrix{
		V \ar[r]_h \ar@/^1pc/[rr]^g & U \ar[r]_f& T
	}
	\]
	and
	\begin{align*}
		g^*(\oomega_{W,A}^f)  & = h^*(f^*(\oomega_{W,A}^f)) =\\
		&=h^*\left( 
		\sum_{\substack{\uno \leq \bj \leq \bm\\ \cap_i H^U_{i, j_i} \subseteq f^{-1}(W)}} \prod_{i \in A} \oomega^U_{i,j_i}\right)
	\end{align*}
	where the multi-index $\bj$ is as in the Remark~\ref{rem:riparametrizzazione}. The last equality follows since $f$ separates $A$.
	
	Again Remark~\ref{rem:riparametrizzazione} applied to $g$ gives
	\[
	g^*(\oomega_{W,A}^g) = \sum_{\substack{\uno \leq \bk \leq \bn\\ \cap_i H^V_{i, k_i} \subseteq g^{-1}(W)}} \prod_{i \in A} \oomega^V_{i,k_i}.
	\]
	where for all $i$, $H^V_{i, 1}, \ldots, H^V_{i,n_i}$ are the connected components of $g^{-1}(H_i)$, $\bk = (k_1, \ldots, k_{|A|})$, and $\bn = (n_1, \ldots, n_{|A|})$.
	
	Now we have that 
	\[
	h^*(\oomega_{i,j_i}^U)= \sum_{h(H^V_{i,k_i}) = H^U_{i,j_i}} \oomega_{i,k_i}^V.
	\]

	Hence from the previous equality we get
	\begin{align*}
		g^*(\oomega_{W,A}^f)  & = \sum_{\substack{\uno \leq \bj \leq \bm\\ \cap_i H^U_{i, j_i} \subseteq f^{-1}(W)}} \prod_{i \in A} h^*(\oomega^U_{i,j_i}) \\
		& = \sum_{\substack{\uno \leq \bj \leq \bm\\ \cap_i H^U_{i, j_i} \subseteq f^{-1}(W)}} \prod_{i \in A} \sum_{h(H^V_{i,k_i}) = H^U_{i,j_i}} \oomega_{i,k_i}^V \\
		& =  \sum_{\substack{\uno \leq \bj \leq \bm\\ \cap_i H^U_{i, j_i} \subseteq f^{-1}(W)}} \sum_{\substack{\uno \leq \bk \leq \bn\\h(H^V_{i,k_i}) = H^U_{i,j_i}}} \prod_{i \in A}  \oomega_{i,k_i}^V \\
		&  = \sum_{\substack{\uno \leq \bk \leq \bn\\ \cap_i H^V_{i, k_i} \subseteq g^{-1}(W)}} \prod_{i \in A} \oomega^V_{i,k_i} \\ 
		& = g^*(\oomega_{W,A}^g).
	\end{align*}
	Finally, in the general case of two coverings $f:U \to T$ and $g:V \to T$, we can consider the diagram
	\[
	\xymatrix{
		& V'\ar[ld]_h \ar[rd] \ar[dd]^{g'}&\\
		U\ar[rd]_f & & V \ar[ld]^g\\
		& T &
	}
	\]
	where $h:V' \to U$ is the pullback of $g$ by $f$ and $g' = f \circ h$. Since $f$ separates $A$, then also $g'$ separates $A$ and we apply the first part of the proof to the maps $f$ and $g'$.
\end{proof}

\begin{remark}\label{rem:pf}
	Since the covering $f:U \to V$ is finite, we have that $\oomega_{W,A}^f = f_* \oomega_A^U(q)$ for any $q \in f^{-1}(W)$ where $f_*$ is the pushforward associated with the covering map $f$.
\end{remark}

Using Theorem~\ref{th:coverings}, we can state the following definition.
\begin{definition}\label{def:omega_WA}
	Given $A \subset E$ independent and given $W$ a connected component of $\cap_{i \in A}H_i$, we define
	\[
	\oomega_{W,A} := \oomega_{W,A}^f
	\]
	and
	\[
	\omega_{W,A} := \omega_{W,A}^f
	\]
	where $f:U \to T$ is any covering that separates $A$.
\end{definition}

\begin{remark} 
	We would like to convince the reader that the definition of the forms $\oomega_{W,A}$ and $\omega_{W,A}$ given above is the most natural choice in order to provide a set of form generating the cohomology of the toric complement.
	
	As seen in \eqref{eq:decompos_seminvarianti}, once  we fix a covering $f:U \to T$ with Galois group $G$, the $G$-module $\Omega^1(M(\A_U))$ has a natural decomposition as a direct sum of semi-invariant modules associated with the characters of $G$.
	The forms defined above can be identified with certain $G$-invariant forms on $M(\A_U)$.
	We have that 
	\[
	\Omega_\lambda^k \Omega_{\lambda'}^{k'} \subseteq \Omega_{\lambda + \lambda'}^{k+k'}.
	\]
	In particular, if the sum of the characters of the factors is the trivial character, we get invariant forms, which correspond to forms on $M(\A)$.
	
	The hypothesis that $f$ separates $A$ guarantees that we obtain enough 
	semi-invariant $1$-forms associated with the hypertori $f^{-1}(H_i)$, for $i \in A,$ 
	in order to obtain $m(A)$ independent invariant classes.
\end{remark}

\begin{lemma}\label{lem:prodsum}
	If $A, A'\subseteq E$ are such that $A\sqcup A'$ is an independent set and $W$, resp.\ $W'$ are a choice of a connected component of $\bigcap_{i\in A} H_i$, resp.\ $\bigcap_{i\in A'} H_i$, we can compute
	\[
	\oomega_{W,A} \oomega_{W',A'}=
	(-1)^{\ell({A,A'})} \sum_{L\in \pi_0(W\cap W')}
	\oomega_{L,A\sqcup A'}.
	\]
\end{lemma}
\begin{proof}
	Consider a covering $f:U\to T$ that separates the independent set  $A\sqcup A'$ (e.g., the one described in \Cref{ex:sepind}). Then, by definition, in order to evaluate the product $\oomega_{W,A} \oomega_{W',A'}$ we consider its pullback $f^*(\oomega_{W,A}) f^*(\oomega_{W',A'})$ which, with \Cref{rem:riparametrizzazione}, equals
	\[
	\Bigg(
	\sum_{\substack{\uno \leq \bj \leq \bm\\ \cap_i H^U_{i, j_i} \subseteq f^{-1}(W)}} \prod_{i \in A} \oomega^U_{i,j_i} \Bigg)
	\Bigg( \sum_{\substack{\uno \leq \bj' \leq \bm'\\ \cap_{i'} H^U_{i', j'_{i'}} \subseteq f^{-1}(W')}} \prod_{i' \in A'} \oomega^U_{i',j'_{i'}}
	\Bigg)
	\]
	\[
	=
	\sum_{\substack{\uno \leq (\bj,\bj') \leq (\bm,\bm') \\ \cap_i H^U_{i, j_i}\cap_{i'} H^U_{i', j'_{i'}} 
			\subseteq f^{-1}(W'\cap W)}} 
	\prod_{i \in A} \oomega^U_{i,j_i}
	\prod_{i' \in A'} \oomega^U_{i,j'_{i'}}
	\]
	\[
	=
	\sum_{\substack{\uno \leq \overline{\bj} \leq (\bm,\bm') \\ \cap_{i\in A\sqcup A'} H^U_{i, \overline{j}_i} 
			\subseteq f^{-1}(W'\cap W)}} 
	(-1)^{\ell({A,A'})} 
	\prod_{i \in A \sqcup A'} \oomega^U_{i,\overline{j}_i},
	\]
	where $\overline{\bj}=(\bj,\bj')$. The latter equals, by definition
	\[
	f^*\Bigg((-1)^{\ell({A,A'})}\sum_{L\in \pi_0(W\cap W')} 
	\oomega_{L,A\sqcup A'}\Bigg)
	\]
	as was to be shown.
\end{proof}

\begin{remark}
	Assume that $\cap_{i \in A} H_i$ is connected and call it $W$. Then, since the identity separates $A$, we have $\oomega_{W,A} = \oomega_A$.
\end{remark}

\begin{definition} \label{df:oetaWAB}
Given $A \subset E$ independent and given $W$ a connected component of 
$\cap_{i \in A}H_i$, we write $\oeta_{W,A,B}$ for the form
\[
(-1)^{\ell({A,B})}\oomega_{W,A} \psi_B.
\]
\end{definition}
In the following if $W$ is not a connected component of $\cap_{i \in A}H_i$  the expression $\oeta_{W,A,B}$ will be considered as meaningless and it will be treated as zero.

\begin{remark} \label{rem:prodrel}
	We have the following immediate consequence of \Cref{lem:prodsum}. 
	If $A, A', B, B'\subseteq E$ are such that $A\sqcup A' \sqcup B \sqcup B'$ is an independent set and $W$, resp.\ $W'$ are a choice of a connected component of $\bigcap_{i\in A} H_i$, resp.\ $\bigcap_{i\in A'} H_i$, we can compute
	\[
	\oeta_{W,A,B} \oeta_{W',A',B'}=
	(-1)^{\ell({A\cup B,A'\cup B'})} \sum_{L\in \pi_0(W\cap W')}
	\oeta_{L,A\sqcup A', B\sqcup B'}.
	\]
\end{remark}

\begin{definition} \label{def:filtrazioneleray}
	We introduce the  
	increasing filtration $\FF$ of $H^*(M(\A);\Z)$ defined by
	\[ \FF_i H^*(M(\A);\Z)  := \sum_{j \leq i} H^j(M(\A);\Z) \cdot H^*(T;\Z). \]
\end{definition}
Such a filtration is the Leray filtration associated with the inclusion $M(\A) \hookrightarrow T.$
The same filtration, with rational coefficients, was introduced in \cite[Remark 4.3.(2)]{DeConciniProcesiToricArr}. 
The associated graded module is
\begin{equation}\label{eq:decompositiongraded}
	\gr_k (H^*(M(\A)))
	= \bigoplus_{
		\substack{W\in \mathcal L(\A)\\
			\operatorname{codim}(W) = k
	}}
	H^*(W) \otimes H^k(M(\A[W]))
\end{equation}
where $\A[W]$ is the hyperplane arrangement introduced in Definition \ref{def:localarrangement}.

\begin{lemma}\label{lem:graduato}
	Let $A,B\subseteq E$ such that $A\sqcup B$ is independent and let $W$ be a connected component of $\cap_{i\in A}H_i$. Then, the image of $\oeta_{W,A,B}$ in $\gr_{\vert A \vert}(H^*(M(\A)))$ equals
	\[
	(-1)^{\ell({B,A})} 2^{\vert A \vert} \psi_B\otimes e_A
	\in H^{\vert B \vert } (W) \otimes H^{\vert A\vert}(M(\A[W])),\]
	where $e_A$ denotes the canonical generator in the top-degree of the Orlik-Solomon algebra of the hyperplane arrangement $\A[W]$ associated with the hyperplanes indexed by $A$ (\cf \Cref{def:localarrangement}).
\end{lemma}
\begin{proof}
	We consider the corresponding graduation $\operatorname{gr}^U$ for the lift to a unimodular covering $f: U\to T$ (e.g., the one separating $A$ in \Cref{ex:sepind}). 
	
	By multiplicativity of $\operatorname{gr}_k$, it suffices to prove the case $B=\emptyset$. We thus have to consider $\oomega_{W,A}$ which, by \Cref{rem:pf}, can be written as $\oomega_{W,A}=f_* \oomega_A^U(q)$, where  $q$ is a fixed point in $f^{-1}(W)$. Now,
	\[\operatorname{gr}^U_k \left( 
	\oomega_A^U(q)\right) = 
	\operatorname{gr}^U_k \left(
	2^{\vert A\vert}\omega_A^U(q)\right),
	\]
	hence 
	\[
	\operatorname{gr}_k(\oomega_{W,A}) = f_* (\operatorname{gr}_k \oomega_A^U(q))
	=\operatorname{gr}_k (2^{\vert A\vert}\omega_{W,A}),
	\]
	and, since $\operatorname{exp}_p^*(\omega_i)=e_i$, the class $[2^{\vert A\vert}\omega_{W,A}]$ maps to the element $2^{\vert A\vert}\otimes e_A$ in $H^{0} (W) \otimes H^{\vert A\vert}(M(\A[W]))$ as desired.
\end{proof}

\section{Unimodular coverings of toric arrangements and rational cohomology}\label{s:rational_cohomology}

Let $\A = \{ H_0, \ldots, H_n\}$ be a primitive, central and essential arrangement in the torus $T$. Suppose further that the associated matroid has exactly one circuit $C\subseteq E$, and hence $\operatorname{rk} E = n$. Let $\chi_0, \ldots, \chi_n$ be the associated list of characters, and suppose that $\chi_0,\ldots,\chi_k$ are the characters associated to the circuit $C$.  

Recall that we identify the ground set of the matroid with indices of the characters, so that $C=\{0,\ldots,k\}$ and $E=\{0,\ldots,n\}$. We let also $F:=E\setminus C$, the letter $F$ (for ``free'') signifying that the elements in $F$ are coloops of the matroid.

Recall from \S \ref{sec:AM} that $\Lambda_E \subset \Lambda$ is the sublattice generated by the characters of $E$ and $\Lambda^E$ is the intersection $(\mathbb Q\otimes_{\mathbb Z} \Lambda_E) \cap \Lambda$.

\begin{definition}
	For every $i=0,\ldots,k$ set 
	\[
	a_i:= m(E) \prod_{j\in C, j \neq i}m(C \setminus \{j\}) .
	\] 
	For every $i=k+1,\ldots,n$ set $a_i = m(E)$.
	We call $\XE$ the lattice in $\Q \otimes_\Z \Lambda$ generated by the elements $\frac{\chi_i}{a_i}$, $i=0,\ldots,n$. 
\end{definition}

\begin{remark}\label{rem:ma} 
	
	Since 
	 $\A$ is essential we have that $\Lambda=\Lambda^E$, hence $m(E)$ is precisely the index of $\Lambda_E$ in $\Lambda$. 
\end{remark}

\begin{lemma}\label{lem:unimodularity}
	In $\XE$ we have the relation
	\begin{equation}\label{eq:circuito_2}
		\sum_{i = 0}^k {c}_i  \frac{\chi_i}{a_i} = 0
	\end{equation}
	where $c_i\in \{+1,-1\}$ for all $i$.
\end{lemma}
\begin{proof} This follows from Lemma~\ref{lem:arcirc}, since
	the product 
	\[a_i m(C \setminus \{i\}) = m(E)^{k+1} \prod_{j=0}^k m(C \setminus \{j\})\]
	 does not depend on the index $i$.
\end{proof}

\begin{lemma}
	The lattice $\XE$ contains $\Lambda$.
\end{lemma}
\begin{proof} We split the claim into two inclusions: 
	\[\Lambda \stackrel{\textrm{(i)}}{\subseteq} \frac{1}{m(E)} \Lambda_E \stackrel{\textrm{(ii)}}{\subseteq} \XE .\]
	
	Inclusion (i) follows from the fact that, by Remark~\ref{rem:ma} the quotient $\Lambda/\Lambda_E$ is a group of cardinality $m(E)$, hence $m(E) \Lambda \subset 
	\Lambda_E$. 
	
	For inclusion (ii), notice that every element of $\frac{1}{m(E)}\Lambda_E$ can be written as a combination
	\[
	\frac{1}{m(E)}\sum_{i=0}^n n_i\chi_i = \sum_{i=0}^n \left(\frac{n_i a_i}{m(E)}\right) 
	\frac{\chi_i}{a_i}
	\]
	for some $n_i\in \mathbb Z$.
	Now, 
	$a_i$ is an integer divisible by $m(E)$ for all $i$. 
	 Hence all parenthesized coefficients on the r.h.s.\ are integers, which means $\frac{1}{m(E)}\Lambda_E \subseteq \XE$, as claimed.
	
\end{proof}

\begin{lemma}\label{lem:degree}
	The inclusion of lattices $\Lambda \subseteq \XE$ induces a covering of $T$ of degree 
	\[
	d=\frac{\prod_{j\in E\setminus \{i\}} a_j}{m(E\setminus \{i\})} 
	\]
	where $i$ is any element of $C$.
\end{lemma}
\begin{proof} It is enough to prove that $d$ as defined above equals the index of $\Lambda$ in $\XE$.
	Let us fix an index $i\in C$ and consider the inclusions
	\[
	\Lambda_{E\setminus \{i\}} \subseteq \Lambda \subseteq \XE.
	\] 
	Since (by Lemma~\ref{lem:unimodularity}) the lattice $\XE$ is generated by the basis $\{ \frac{\chi_j}{a_j} \mid j \neq i\}$, the index of $\Lambda_{E \setminus \{i\}}$ in $\XE$ is
	\[
	[\XE:\Lambda_{E \setminus \{i\}}] = \prod_{j \neq i} a_j.
	\]
	On the other hand, $m(E \setminus \{i\})$ is by definition the index of $\Lambda_{E \setminus \{i\}}$ in $\Lambda = \Lambda^{E \setminus \{i\}}$. 
	In conclusion, the desired index is
	\begin{align*}
		[\XE:\Lambda] &=\frac{[\XE:\Lambda_{E \setminus \{i\}}]}{[\Lambda:\Lambda_{E \setminus \{i\}}]}
		=
		\frac{\prod_{j \neq i} a_j}{m(E \setminus \{i\})} 
	\end{align*}
and the claim follows.
\end{proof}

\begin{definition} Let 
	\[\pi_U: U\to T\]
	denote the covering induced by the inclusion $\Lambda \subseteq \XE$.
	
	We denote by $\A_U$ the central arrangement in the torus $U$ induced by the characters $\frac{\chi_i}{a_i}$ in $\XE$. Notice that $\A_U$ is clearly primitive, since the $\frac{\chi_i}{a_i}$ contain a basis of $\XE$.
\end{definition}

\begin{lemma}
	The arrangement $\A_U$ is unimodular.
\end{lemma}
\begin{proof} For every $j \in C$ the set
$
\{\frac{\chi_i}{a_i}\}_{i \neq j} $ is a basis of the lattice $\Lambda(E)$.
In fact, $E \setminus \{j\}$ is independent and 
by  \eqref{eq:circuito_2}  we have that 
$\frac{\chi_j}{a_j}$ belongs to the lattice generated by the characters $\frac{\chi_i}{a_i}$, $i\in C\setminus \{j\}$.

Hence for every independent subset $A \subseteq E$  
 we can choose $j \in C \setminus A$. Then 
the set $\{\frac{\chi_i}{a_i}\}_{i \in A}$ can be completed to a basis 
$\{\frac{\chi_i}{a_i}\}_{i \in E\setminus \{j\}}$ of 
$\Lambda(E)$. 
\end{proof}

Notice that the number of connected component of $\pi_U^{-1}(H_i)$ is $a_i$. In fact for $j \neq i$ the character ${\chi_i}$ can be written in the basis $\{\frac{\chi_k}{a_k}\}_{k \neq j}$  as 
$
\chi_i = a_i \frac{\chi_i}{a_i}
$.

\begin{lemma}\label{lem:WAE} Let $A\subsetneq E$ be independent,  let $W$ be  a connected component of $\bigcap_{i\in A} H_i$ and choose $p\in W$. Then, for every layer $L$ of $\A_U$ such that $\pi_U(L)=W$, the number of preimages of $p$ contained in $L$ is 
	\[|L \cap \pi_U^{-1}(p)| = 
	\frac{m(A)}{m(E\setminus\{j\})} \prod_{i\in E\setminus (A\cup \{j\})} a_i 
	\]
	where $j$ is any element of $C\setminus A$.
\end{lemma}
\begin{proof}
	The cardinality of the preimage of $p$ is equal to the degree of the covering, computed in Lemma~\ref{lem:degree}. 
	On the other hand, 
	given $W$ a connected component of $\cap_{i\in A} H_i$, the number of connected components of $\pi_U^{-1}(W)$ is equal to $\frac{\prod_{i \in A}a_i}{m(A)}$. Hence, for any $j\in C\setminus A$, 
	\begin{align*}
		|L \cap \pi_U^{-1}(p)| & = \frac{m(A) \prod_{i\in E\setminus \{j\}} a_i}{m(E\setminus \{j\})\prod_{i \in A}a_i}\\
		&  = 
		\frac{m(A)}{m(E\setminus\{j\})} \prod_{i\in E\setminus (A\cup \{j\})} a_i . \qedhere
	\end{align*}
\end{proof}
\begin{example}
	In the case of the arrangement of Example \ref{ex:cover} 
	with matrix 
	\[
	\left(
	\begin{array}{ccc}
	3 & 0 & 1 \\
	1 & 1 & 0 
	\end{array}
	\right)
	\]
	we have $E=C$ and the lattice $\Lambda = \Z^2$ coincides with the lattice $\Lambda^C$. In this case we have $a_0 = 3, a_1 = 3, a_2 = 1$, hence the lattice $\Lambda(C)$ is generated by $\langle  e_1 +\frac{e_2}{3},  \frac{e_2}{3} , e_1 \rangle.$ In particular the inclusion 
	$
	\Lambda \subset \Lambda(C)
	$
	corresponds to the covering 
	$f \colon U \to T$
	of Example \ref{ex:cover} (see Figure \ref{fig:covering}). Notice that, with respect to the basis $\{ e_1, \frac{e_2}{3}\}$ of $\Lambda(C)$, the arrangement $\exA_U$ is described by the matrix 
	\[
	\left(
	\begin{array}{ccc}
	3 & 0 & 1\\
	3 & 3 & 0
	\end{array}
	\right).
	\]
\end{example}

\begin{lemma} \label{lem:pull_back_oeta}
	For any $A,B\subseteq E$ such that $A\sqcup B$ is a maximal independent subset of $E$, for any connected component $W$ of $\bigcap_{i\in A}H_i$ and $p\in W$  we have
	
	\begin{equation}\label{eq:oeta_new}
		\pi_U^*(\oeta_{W,A,B}) =   (-1)^{\ell({A,B})} \frac{m(A \cup B)}{m(A)}  \sum_{q \in \pi_U^{-1}(p)}  \oomega^U_A(q) \psi^U_B.
	\end{equation}
	
\end{lemma}

\begin{proof}
	With \Cref{df:psiU} we have
	\[
	\pi_U^*(\psi_B) = \left( \prod_{i \in B}  a_i \right) \psi_B^U.
	\]
	and hence, applying this equality and Remark~\ref{rem:fstaromega} to Definition~\ref{df:oetaWAB}, we get
	\begin{equation}\label{eq:oeta}
		\pi_U^*(\oeta_{W,A,B}) = (-1)^{\ell({A,B})} \frac{\prod_{i \in B}  a_i}{|L \cap \pi_U^{-1}(p)| }  \sum_{q \in \pi_U^{-1}(p)}  \oomega^U_A(q) \psi^U_B .
	\end{equation}
	
	The coefficient in formula \eqref{eq:oeta} can be rewritten using \Cref{lem:WAE} applied to the independent subset $A\sqcup B$ of $E$ (noticing that there is a unique $j\in E\setminus (A\sqcup B)$, so that in fact $E\setminus (A\cup\{j\})=B$ and $E\setminus \{j\}=A\sqcup B$). We obtain
	\begin{align*}
		\frac{\prod_{i \in B} a_i}{|L \cap \pi_U^{-1}(p)| } & =  \frac{m(A\cup B)}{m(A)} \\
	\end{align*}
	and the claim follows.
\end{proof}
\begin{definition}
	For any $A,B\subseteq E$ such that $A\sqcup B$ is an independent set and every $q \in \pi_U^{-1}(\bigcap_{i\in A}H_i)$,
	we set
	\[\oeta_{A,B}^U(q):=(-1)^{\ell({A,B})}\oomega_A^U(q) \psi^U_B.\]
\end{definition}

Recall from \Cref{prop:rel_unimod} that given a circuit $C$, for every $B \subset C$ we set $c_B= \prod_{i \in B} c_i$.

\begin{theorem} \label{thm:RelCirc}
	Let $E= C \sqcup F$ as above and let $L$ be a connected component of $\cap_{i \in E}H_i$. We have
	\begin{equation}
		\sum_{j\in C} \sum_{
			\substack{
				A,B\subset E\\
				A \supseteq F \\
				C= A \sqcup B \sqcup \{j\}\\
				\vert B \vert \textnormal{ even}\\
				W\supseteq L
		}} 
		(-1)^{|A_{\leq j}|} c_B \frac{m(A)}{m(A \cup B)} \oeta_{W,A,B}  = 0.
	\end{equation}
\end{theorem}

\begin{proof}
	Now we fix a point $p \in L \subseteq \bigcap_{i\in E} H_i$ and we use relation~\eqref{eq:relazione_bar} in $\A_U$ for the circuit $C$. This gives us, for every $q \in \pi_U^{-1} (p)$, 
	\[ 
	\sum_{j\in C} \sum_{
		\substack{
			A,B\subset C\\
			C= A \sqcup B \sqcup \{j\}\\
			\vert B \vert \textnormal{ even}
	}}
	(-1)^{|A_{\leq j}|} \oeta^U_{A,B}(q) c_B = 0.
	\]
	We multiply this equation by $\oeta^U_{F, \emptyset}(q)$, and using the equality $\oeta^U_{A,B}(q)\oeta^U_{F, \emptyset}(q) = (-1)^{|F_{\leq j}|+\ell(C,F)} \oeta^U_{A\sqcup F,B}(q)$ we obtain:
\[
\sum_{j\in C} \sum_{
	\substack{
		A,B\subset E\\
		A \supseteq F \\
		E= A \sqcup B \sqcup \{j\}\\
		\vert B \vert \textnormal{ even}
}}
(-1)^{|A_{\leq j}|} \oeta^U_{A,B}(q) c_B = 0.
\]
	Summing over all $q \in \pi_U^{-1} (p)$, we get
	\begin{align*}
		0 & =\sum_{q \in \pi_U^{-1} (p)}
		\sum_{j\in C} \sum_{
			\substack{
				A,B\subset E\\
				A \supseteq F \\
				E= A \sqcup B \sqcup \{j\}\\
				\vert B \vert \textnormal{ even}
		}}
		(-1)^{|A_{\leq j}|} c_B  \oeta^U_{A,B}(q)\\
		& =
		\sum_{j\in C} \sum_{
			\substack{
				A,B\subset E\\
				A \supseteq F \\
				E= A \sqcup B \sqcup \{j\}\\
				\vert B \vert \textnormal{ even}
		}}
		(-1)^{|A_{\leq j}|} c_B\sum_{q \in \pi_U^{-1} (p)} \oeta^U_{A,B}(q)  \\
		& = 
		\sum_{j\in C} \sum_{
			\substack{
				A,B\subset E\\
				A \supseteq F \\
				E= A \sqcup B \sqcup \{j\}\\
				\vert B \vert \textnormal{ even}\\
				W\supseteq L
		}}
		(-1)^{|A_{\leq j}|}  c_B \frac{m(A)}{m(A \cup B)} \pi^*_U(\oeta_{W,A,B}),
	\end{align*}
	where in the last equality we use eq.~\eqref{eq:oeta_new}.
	Since $\pi_U^*$ is an injective algebra homomorphism, we obtain the claimed equality.
\end{proof}

We now drop the assumption that the arithmetic matroid has a unique circuit and we go back to the general set-up of any arrangement $\A$ in a torus $T$.

\begin{remark}
We have constructed a unimodular covering for an arrangement with only one circuit.
Notice that this construction cannot be done globally, i.e.\ for a general toric arrangement.
Indeed consider the central toric arrangement $\A$ in $(\C^*)^2$ described by the matrix
\[ \left( \begin{array}{cccc}
1 & 0 & 2 & 1 \\ 
0 & 1 & 1 & 3
\end{array} \right),\]
does not exist a covering $f \colon U \to (\C^*)^2$ such that the toric arrangement $\A_U = \bigcup_{H \in \A} \pi_0(f^{-1}(H))$ is unimodular.
\end{remark}

\begin{theorem} \label{thm:main_rational}
	Let $\A$ be an essential arrangement. 
	The rational cohomology algebra $H^*(M(\A),\mathbb Q)$ is isomorphic to the algebra $\mathcal E$ with 
	\begin{itemize}
		\item[--] Set of generators $e_{W,A;B}$, where $W$ ranges over all layers of $\A$, $A$ is a set generating $W$ and $B$ is disjoint from $A$ and such that $A\sqcup B$ is an independent set;
		the degree of the generator $e_{W,A;B}$ is $\vert A \sqcup B\vert$.
		\item[--] The following types of relations
		\begin{itemize}
			\item For any two generators $e_{W,A;B}$, $e_{W',A';B'}$,  
			\[e_{W,A;B} e_{W',A';B'}=0\]
			 if $ A \sqcup B \sqcup A' \sqcup B' $  is a dependent set, and otherwise 
			\begin{equation}\label{eq:relazione_prodotto}
				e_{W,A;B}e_{W',A';B'}= 
				(-1)^{\ell(A\cup B,A'\cup B')} 
				\sum_{L\in \pi_0(W\cap W')}
				e_{L,A\cup A';B\cup B'}.
			\end{equation}
\item For every linear dependency $\sum_{i\in E} n_i \chi_i=0$ with $n_i\in \mathbb Z$, a relation
\begin{equation}\label{eq:relazione_toro}
\sum_{i\in E} n_i e_{T,\emptyset;\{i\} } =0.
\end{equation}
			
\item  For every subset $X\subseteq E$ where $\rk(X)=\vert X \vert -1$ write $X=C\sqcup F$ with $C$ the unique circuit in $X$. Consider the  associated (unique) linear dependency $\sum_{i\in C} n_i\chi_i=0$ with $n_i\in \mathbb Z$, and for every connected component $L$ of $\cap_{i \in X} H_i$ a relation
\begin{equation}\label{eq:relazione_final}
\sum_{j\in C} \sum_{
					\substack{
						A,B\subset X\\
						A\supseteq F\\
						X= A \sqcup B \sqcup \{j\}\\
						\vert B \vert \textnormal{ even}\\
						W\supseteq L
				}}
				(-1)^{|A_{\leq j}|} c_B  \frac{m(A)}{m(A \cup B)} e_{W,A;B} = 0
			\end{equation}
			where, for all $i\in C$, $c_i:=\operatorname{sgn}n_i$, $c_B = \prod_{i \in B}c_i$.
			
		\end{itemize}
	\end{itemize}     
\end{theorem}

\begin{proof}
	
	Consider the map
	\[
	\Phi: \mathcal E \to H^*(M(\A),\mathbb Q),\quad
	e_{W,A;B} \mapsto [\oeta_{W,A,B}].
	\]
	This map is well-defined -- in fact, in the cohomology ring Equation \eqref{eq:relazione_prodotto} holds by \Cref{rem:prodrel}, Equation \eqref{eq:relazione_toro} already holds in the cohomology of the ambient torus, and Equation \eqref{eq:relazione_final} holds by \Cref{thm:RelCirc}. 
	
We claim that each generator $e_{W,A;B}$ can be expressed as linear combination of generators $e_{W',A';B'}$ with $A'$ a no-broken-circuit for the arrangement $\A[W']$. Indeed if $A$ contains a broken-circuit $A_1 \subseteq A$, i.e.\ $A_1=C \setminus \min (C)$ for a circuit $C$, consider the relation \eqref{eq:relazione_final} for $X= C\cup A$: it expresses the element $e_{W,A;B}$ as a linear combination of some $e_{W',A';B'}$ with $|A'|<|A|$ or with $|A'|=|A|$ and $A'$ lexicographically bigger than $A$. We have inductively proved the claim.
	
	Now fix, for every independent $A\subseteq E$, a subset $D(A)\subseteq E$ such that $A\sqcup D(A)$  is a basis of the matroid.
Then, notice that relations \eqref{eq:relazione_prodotto} and  \eqref{eq:relazione_toro} allow us to express every generator $e_{W,A;B}$, with $A$ no-broken-circuit, in terms of generators $e_{W,A;B}$ where $B$ is a subset of $D(A)$.
	Then, with \Cref{lem:graduato}, the $k$-th graded part of the image of $\Phi$ equals $\operatorname{gr}_k H^*(M(\A),\mathbb Q).$ We conclude that $\Phi$ is bijective, hence it defines the desired isomorphism.
\end{proof}
\begin{remark}
	The relations in the presentation above hold for differential forms and not only for their cohomology classes. As a consequence the space $M(\A)$ is rationally formal. This fact has been already observed by \cite{DeConciniProcesiToricArr} for unimodular arrangements and proved by \cite{DupontFormality2016} in general.
\end{remark}

\begin{remark}\label{combdata} 
	Notice that all relations of type \eqref{eq:relazione_toro} are implied by those associated with minimal linear dependencies (i.e., circuits). 
	
	Moreover, the above presentation is completely encoded in the datum of the poset of layers of $\A$ (needed, e.g., for Relations \eqref{eq:relazione_prodotto}, \eqref{eq:relazione_final}) and in the (relative) sign pattern of the minimal linear dependencies. But by \cite[Theorem 3.12]{Pagaria2017} (see \Cref{rm:cipl}), the latter can also be recovered by the poset.

	The complements of the two toric arrangements constructed in the already quoted paper \cite{pagaria2} (see \Cref{rem:matroid_vs_poset})
	turn out to have non-isomorphic cohomology rings. Since the two arrangements have isomorphic matroids, this implies that the cohomology ring cannot be determined purely in terms of the arithmetic matroid.
\end{remark}

\begin{example} \label{es:formulebar} We can provide a presentation of the rational cohomology of the complement of arrangement  $\exA$.
	
	The cohomology ring is generated by:
	\begin{align*}
		\oomega_0=& \frac{1}{2 \pi \sqrt{-1}}\dlog\left(\frac{(1-x^3y)^2}{x^3y}\right),& \psi_0= &\frac{1}{2 \pi \sqrt{-1}} \dlog(x^3y),\\
		\oomega_1=& \frac{1}{2 \pi \sqrt{-1}}\dlog\left(\frac{(1-y)^2}{y}\right),&
		\psi_1= & \frac{1}{2 \pi \sqrt{-1}} \dlog(y),\\
		\oomega_2=& \frac{1}{2 \pi \sqrt{-1}}\dlog\left(\frac{(1-x)^2}{x}\right),&  \psi_2= & \frac{1}{2 \pi \sqrt{-1}} \dlog(x)
	\end{align*}
	and
	\begin{align*}
		\oomega_{p, \{0,1\}} = & \frac{-1}{4\pi^2}\frac{x^3 y^2+x^3y +4x^2y+4xy+y +1}{xy \left(y-1\right) \left(x^3y-1\right)} \dd x \dd y,\\
		\oomega_{q, \{0,1\}} = & \frac{-1}{4\pi^2}\frac{x^3 y^2+x^3y +4\zeta_3^2 x^2y+4\zeta_3xy+y +1}{xy \left(y-1\right) \left(x^3y-1\right)} \dd x \dd y,\\
		\oomega_{r, \{0,1\}} = & \frac{-1}{4\pi^2}\frac{x^3 y^2+x^3y +4\zeta_3x^2y+4\zeta_3^2xy+y +1}{xy \left(y-1\right) \left(x^3y-1\right)} \dd x \dd y
	\end{align*}
	where $\zeta_3 = e^{\frac{2 \pi i }{3}}$.
	The relations are
	\begin{align*}
		\oomega_i \psi_i =& 0 \qquad \forall \, i, \\
		\oomega_0 \oomega_1 =& \oomega_{p, \{0,1\}} + \oomega_{q, \{0,1\}} + \oomega_{r, \{0,1\}},\\
		-\psi_0 + \psi_1 + 3 \psi_2 =& 0,\\
		\oomega_{p, \{0,1\}} - \oomega_0 \oomega_2 + \oomega_1 \oomega_2 = & - \psi_1 \psi_2 + \psi_0\psi_2 +\frac{1}{3} \psi_0 \psi_1
	\end{align*}
	where the last relation can be verified checking the equalities
	\begin{align*}
		& \frac{x^3 y^2+x^3y +4x^2y+4xy+y +1}{xy \left(y-1\right) \left(x^3y-1\right)} \dd x \dd y -
		\frac{x^3y+1}{y(x^3y-1)} 
		\frac{x+1}{x(x-1)}   \dd y \dd x +\\
		&
		+ \frac{y+1}{y(y-1)}\frac{x+1}{x(x-1)} \dd 
		y \dd x 
		=	\frac{\dd x \dd y}{xy} = \\
		& =  - \dlog (y) \dlog (x) +  \dlog (x^3y) \dlog (x) +{\frac{1}{3}} \dlog (x^3y) \dlog (y) .
	\end{align*}
\end{example}

\begin{example}
Consider the central toric arrangement $\mathcal{C}$ in $T= (\C^*)^3$ given by the four hypertori $H_1=\{x=1\}$, $H_2=\{y=1\}$, $H_3=\{xy=1\}$, and $H_4=\{xy^{-1}z^3=1\}$.
The zero-dimensional layers are the points $p=(1,1,1)$, $q=(1,1, \zeta_3)$, and $r=(1,1,\zeta_3^2)$.
Let $W$ be the layer $H_1 \cap H_2 \cap H_3$.
The relations given by \eqref{eq:relazione_final} are the following:
\begin{align*}
& \oomega_{W,\{1,2\}}-\oomega_{W,\{1,3\}}+\oomega_{W,\{2,3\}} + \psi_{1}\psi_{2} - \psi_{1}\psi_{3} - \psi_{2}\psi_{3} = 0 \\
& \oomega_{s,\{1,2,4\}}-\oomega_{s,\{1,3,4\}}+\oomega_{s,\{2,3,4\}} + \frac{1}{3}\psi_{1}\psi_{2}\oomega_4 - \frac{1}{3}\psi_{1}\psi_{3}\oomega_4 - \frac{1}{3}\psi_{2}\psi_{3}\oomega_4 =0
\end{align*}
for all $s=p,q,r$.
Notice that the relation in degree two does not imply the relations in degree three.
\end{example}

\section{Integral cohomology}\label{sec:integer}

\begin{proposition}
	The forms $\omega_{W,A}$ are integral forms.
\end{proposition}
\begin{proof}
	We first prove our statement in the case when $W$ is a point, hence $\vert A \vert =n$ and $\omega_{W,A}$ is a $n$-form. 
	
	We will prove that for any integral cycle $S \in H_{n}(M(\A); \Z)$ the integral \[\int_S \omega_{W,A}\] is an integer number. From the Universal Coefficients Theorem this implies that $\omega_{W,A}$ is an integral form.
	
	Now, let $f: U \to T$ be any covering that separates $A$ and such that the arrangement $\A_U$ is unimodular. For any cycle $S \in H_{n}(M(\A); \Z) $ we have
	\begin{align*}
		\int_S \omega_{W,A} & = \frac{1}{\deg f} \int_{f^{-1}(S)} f^*(\omega_{W,A})  \\
		& = \frac{1}{\deg f} \int_{f^{-1}(S)}\sum_{q \in f^{-1}(W) } \omega^U_A(q)
	\end{align*}
	where the last equality follows from Remark \ref{rem:fstaromega}.
	
	Now we can observe that the integral
	\[
	\int_{f^{-1}(S)} \omega^U_A(q)
	\]
	does not depend on the point $q \in f^{-1}(W)$. Moreover, since $W$ is a point, $\deg f = | f^{-1}(W)|$. We thus have
	\[
	\int_S \omega_{W,A} = \int_{f^{-1}(S)} \omega^U_A(q)
	\]
	for any point $q \in f^{-1}(W)$.
	Since the arrangement $\A_U$ is unimodular we have (see \eqref{eq:prodotti_omega}) $\omega^U_A(q) = \prod_{i \in A} \omega_i^U(q).$  By definition \eqref{eq:def_omega_U} each factor $\omega_i^U(q)$ is an integer form.	
	Hence integrality of $\omega^U_A(q)$ implies integrality of $\omega_{W,A}$.

	In the general case let $W_0$ be the translate of $W$ containing the identity of $T$. We can consider the projection $\pi_{W_0}: T \to T'$, where $T' = T/{W_0}$. The $W_0$-invariant characters $\chi_i$ for $i \in A$ induce characters $\chi'_i$ of $T'$, defining hypertori $H_i'= \pi_{W_0}(H_i) \subseteq T'$. Let $W'= \pi_{W_0}(W)$ be the component of $\bigcap_{i \in A} H_i'$ corresponding to $W$ and consider the associated form $\omega'_{W',A}$ on $T'$. Then $\omega_{W,A} = \pi^*_{W_0}(\omega'_{W',A})$ and so integrality of $\omega_{W,A}$ follows from integrality of $\omega'_{W',A}$, which is granted because $W'$ has dimension $0$.
\end{proof}

\begin{proposition}
	For any independent set $A\sqcup B \subset E$ and for any layer $W$ in $\cap_{i \in A}H_i$ the form $\frac{m(A)}{m(A \sqcup B)}\eta_{W,A,B}$ is integral.
\end{proposition}
\begin{proof}
	Let $\mathfrak{A} = \{b_1, \ldots, b_{|A|} \}$ be a basis of $\Gamma^A$. We complete it to a basis $\mathfrak{A} \cup \mathfrak{B} = \{b_1, \ldots, b_{|A|+|B|}\}$ of $\Gamma^{A \sqcup B}$. We can define the forms \[v_j:= \frac{1}{2 \pi \sqrt{-1}}\dd \log(e^{b_j}).\]
	Hence we can consider the square matrix  $M=(m_{ij})$ such that for every $j \in A \sqcup B$ we have that \[\psi_j = \sum_{i=1}^{|A|+|B|} m_{ij} v_j.\]
	The matrix $M$ is a
	block matrix of the form
	\[
	M = \left(
	\begin{matrix}
	M_1 & M_2 \\0 & M_3
	\end{matrix}
	\right)
	\]
	with $M_1$ a $|A| \times |A|$ matrix and $M_3$ a $|B| \times |B|$ matrix.
	For $j > |A|$ we have that $\omega_{W,A} \psi_j = \omega_{W,A} \sum_{i=|A|+1}^{|A|+|B|} m_{ij} v_j$, i.e., using only entries of $M_3$.
	Hence 
	\[\eta_{W,A,B} = \pm \omega_{W,A} \prod_{j = |A|+1}^{|A|+|B|} \psi_j = \pm \det(M_3) \omega_{W,A}  \prod_{j = |A|+1}^{|A|+|B|}  v_j.
	\]
	Since $\det(M_3) = \det(M)/\det(M_1) = m(A \sqcup B)/m(A)$
	we have that  $\frac{m(A)}{m(A \sqcup B)}\eta_{W,A,B}$ is an integral form.
\end{proof}

Recall the filtration $\FF$ of $ H^*(M(\A))$ introduced in \Cref{def:filtrazioneleray}:
\[ \FF_i H^*(M(\A);\Z)  := \bigoplus_{j \leq i} H^j(M(\A);\Z) \cdot H^*(T;\Z). \]
From Definition \ref{def:omega_oomega_psi} we have that
\[
[\oomega_i] = 2[\omega_i] \mbox{ in } \FF_1/\FF_0 H^*(M(\A); \Z)
\]
and
\begin{equation}\label{eq:omegaoomega}
	[\oomega_{W,A}] = 2^{|A|}[\omega_{W,A}] \mbox{  in } \FF_{|A|}/\FF_{|A|-1} H^*(M(\A); \Z)
\end{equation}

\begin{definition}
	The $\Z$-algebra  $R \subset \Omega^*(M(\A))$ is the subalgebra generated by the closed forms  $\omega_{W,A} \alpha$, where $W$ runs among the layers of $\cap_{i \in A} H_i$ for $A$ independent and $\alpha \in H^*(T; \Z)$.
\end{definition}

\begin{theorem}\label{thm:main_integral}
	Let $\A$ be an essential toric arrangement. The integral cohomology ring of $M(\A)$ is isomorphic to the algebra $R$:
	\[R \simeq H^*(M(\A); \Z).\]
	In particular the space $M(\A)$ is formal.
\end{theorem}
\begin{proof}
	Since the relations given in the presentation of \Cref{thm:main_rational} are equalities between differential forms, the map  $i:R \hookrightarrow H^*(M(\A); \Z)$ sending each form to its cohomology class is an injective map of filtered modules. 
	In particular it induces an homomorphism
	\[
	\gr(i):\gr(R) \to \gr (H^*(M(\A); \Z)) 
	\]
	of graded modules. We claim that the map  $\gr(i)$ is an isomorphism. Since 
	the strictly filtered map $i$ is injective, $\gr(i)$ is injective too. 
	
	We will prove that $\gr(i)$ is also surjective.
	As seen in \Cref{eq:decompositiongraded}, the graded algebra decomposes as a direct sum
	\[
	\gr_k (H^*(M(\A)))
	= \bigoplus_{
		\substack{W\in \mathcal L(\A)\\
			\operatorname{codim}(W) = k
	}}
	H^*(W) \otimes H^k(M(\A[W]));\]
	moreover the summand $H^*(W) \otimes H^k(M(\A[W]))$ is generated, as a $H^*(T;\Z)$-module, by the elements 
	$
	1 \otimes e_A
	$ for $A$ independent set such that $W$ is a connected component of $\bigcap_{i \in A} H_i$. From \Cref{eq:omegaoomega} and \Cref{lem:graduato} we have
	\[
	2^{|A|} [\omega_{W,A}] = [\oomega_{W,A}] = 2^{|A|}(1 \otimes e_A)_W
	\]
	where the inclusion of $H^*(W) \otimes H^k(M(\A[W]))$ in $\gr_k (H^*(M(\A)))$ is understood.

	Since the integral cohomology ring of the complement of an hyperplane arrangement is torsion free \cite{OrlikSolomonOriginal}, it follows  that the algebra $\gr_k (H^*(M(\A)))$ is torsion free.
	
	For every layer $W$ and every set of indices $A$  the element \[(1 \otimes e_A)_W \in \gr_k (H^*(M(\A)))\] is the image of $\omega_{W,A}$ and hence $\gr(i)$ is surjective. 
	Since $\gr(i)$ is an isomorphism, the claim follows.
\end{proof}

\begin{proposition}\label{prop:relazioni_bar_senzabar}
	The generators $\eta_{W,A,B}$ can be expressed in terms of the generators of the ring $R = H^*(M(\A); \Z)$ as follows:
	\[
	\oeta_{W,A,B} = \sum_{C \subseteq A}  (-1)^{|C|}2^{|A \setminus C|}\frac{m(A \setminus C)}{m(A)} \eta_{L,A \setminus C, B \cup C}
	\]
	where $L$ is the unique connected component of $\cap_{i \in A \setminus C} H_i$ such that $W \subset L$.
\end{proposition}
\begin{proof}	
	Take any covering $f: U \to T$ that separates $A$, e.~g. the one given in \Cref{ex:sepind}. From \Cref{def:omega_oomega_psi}, 
	\Cref{def:oomegaf}, 
	\Cref{lem:pull_back_oeta} and \Cref{def:omega_WA} it follows that
	\begin{align*}
		f^*(\oeta_{W,A,B}) & = (-1)^{\ell(A,B)} \frac{m(A \cup B)}{m(A)} \sum_{q \in f^{-1}(p)} \oomega_{A}^U(q) \psi^U_B = \\
		& =  (-1)^{\ell(A,B)} \frac{m(A \cup B)}{m(A)} \sum_{q \in f^{-1}(p)} \prod_{i \in A} (2 \omega_i^U(q) - \psi_i^U) \psi^U_B =
		\\
		& =  \frac{m(A \cup B)}{m(A)} \sum_{
			\substack{
				q \in f^{-1}(p)\\
				C \subseteq A
			}
		} 
		(-1)^{\ell(A,B) +\ell(A \setminus C,C)+|C|}2^{|A \setminus C|} \omega_{A \setminus C}^U(q) \psi_{C}^U \psi_B^U=\\
		&= 
		f^*\left(
		\sum_{C \subseteq A}
		(-1)^{|C|} 2^{|A \setminus C|}
		\frac{m(A \setminus C)}{m(A)}
		\eta_{L,A \setminus C, B \cup C}
		\right)
	\end{align*}
	where $L$ is the unique connected component of $\cap_{i \in A \setminus C} H_i$ containing $W$. The equality follows from the injectivity of the pull-back map.
\end{proof}

\begin{example}
	For the arrangement $\exA$ the previous relation gives the following presentations for the integral cohomology. We can take as generators the forms
	\begin{align*}
		\omega_0=& \frac{1}{2 \pi \sqrt{-1}}\dlog(1-x^3y),& \psi_0= &\frac{1}{2 \pi \sqrt{-1}} \dlog(x^3y),\\
		\omega_1=& \frac{1}{2 \pi \sqrt{-1}}\dlog(1-y),&
		\psi_1= & \frac{1}{2 \pi \sqrt{-1}} \dlog(y),\\
		\omega_2=& \frac{1}{2 \pi \sqrt{-1}}\dlog(1-x),&  \psi_2= & \frac{1}{2 \pi \sqrt{-1}} \dlog(x)
	\end{align*}
	and
	\begin{align*}
		\omega_{p, \{0,1\}} = & \frac{-1}{4\pi^2}
		\frac{x^2 y+x +1}{\left(y-1\right) \left(x^3y-1\right)} \dd x \dd y,\\
		\omega_{q, \{0,1\}} = & \frac{-\zeta_3}{4\pi^2}\frac{\zeta_3^2x^2 y+\zeta_3x +1}{\left(y-1\right) \left(x^3y-1\right)} \dd x \dd y,\\
		\omega_{r, \{0,1\}} = & \frac{-\zeta_3^2}{4\pi^2}\frac{\zeta_3x^2 y+\zeta_3^2x +1}{\left(y-1\right) \left(x^3y-1\right)} \dd x \dd y
	\end{align*}
	and we have the equivalent for the relations obtained for the rational cohomology:
	\begin{align*}
		\omega_i \psi_i =& 0 \qquad \forall \, i, \\
		\omega_0 \omega_1 =& \omega_{p, \{0,1\}} + \omega_{q, \{0,1\}} + \omega_{r, \{0,1\}},\\
		-\psi_0 + \psi_1 + 3 \psi_2 =& 0,\\
		\omega_{p, \{0,1\}} - \omega_0 \omega_2 + \omega_1 \omega_2 = &
		-\omega_0 \psi_2.
	\end{align*}
	As an example of application of \Cref{prop:relazioni_bar_senzabar} we can write the relation
	\[
	\oomega_{p,\{0,1\}} = 4 \omega_{p, \{0,1\}} - \frac{2}{3} (\psi_0 \omega_1 + \omega_0 \psi_1) + \frac{1}{3} \psi_0 \psi_1,
	\]
	that can be also checked directly using the formulas above and the formulas in \Cref{es:formulebar}.
\end{example}

\bibliography{biblio_tos}{}
\bibliographystyle{amsalpha}
\end{document}